\documentclass[a4paper]{article}
\usepackage{graphicx} % Required for inserting images
\usepackage[hidelinks]{hyperref}
\usepackage{amsfonts,amssymb,amsmath,mathtools,amsthm}
\usepackage{tabularx}
\usepackage[colorinlistoftodos]{todonotes}
\usepackage{pdfpages}
\usepackage{geometry}
\usepackage{comment}

\theoremstyle{plain}
\newtheorem{theorem}{Theorem}[section]
\newtheorem{proposition}[theorem]{Proposition}

\theoremstyle{definition}

\newtheorem{remark}[theorem]{Remark}
\theoremstyle{plain}

\numberwithin{equation}{section}

\newcommand \Tcal           {\mathcal{T}}

\newcommand \RR           {\mathbb{R}}

\definecolor{other_green}{RGB}{20,130,0}

\begin{document}

% \title{Dengue fever modeling with virus load and antibody levels}
\title{A structured model of vector-borne disease with within-host viral load and antibody dynamics}
\author{Paulo Amorim$^{1}$, M.~Soledad Aronna$^{1}$, Débora O. Medeiros$^{1,2,3}$}
\footnotetext[1]{Escola de Matemática Aplicada, Fundação Getulio Vargas FGV EMAp, Praia de Botafogo 190, Rio de Janeiro 22250-900 RJ Brazil}
\footnotetext[2]{Current address: Instituto de Matemática e Estatística, Universidade Federal Fluminense IME-UFF, São Domingos, Niterói 24210-200 RJ Brazil}
\footnotetext[3]{Corresponding author. \texttt{deboramedeiros@id.uff.br}}

\date{}

%
%
%
%%%%%%%%%%%%%%%%%%%%%%%%%%%%%%%%%%%%%%%%%%%%%%%%%%%%%%%%%%%%%%%%%%

%
\maketitle
%
%\tableofcontents

%
%
%
%%%%%%%%%%%%%%%%%%%%%%%%%%%%%%%%%%%%%%%%%%%%%%%%%%%%%%%%%%%%%%%%%%

\begin{abstract}
    % We present an epidemiological model for vector-borne diseases, including witihin-host viral load and antibody dynamics, via structured (transport) equations, with particular emphasis on Dengue fever. The within-host dynamics is present in the infected and recovered host compartments. The structure induces nonlinearities and nonlocalities in the formulation. We analyse the model, showing well-posedness, mass conservation, and study the characteristic curves of the transport equation. We deduce a simplified Uniform Host Response (UHR) model, which includes delay-type terms. For both the full and UHR model, we determine the basic reproduction number $\mathcal{R}_0$, show that it acts as a threshold value with respect to the existence of an endemic equilibrium, and comment on its relation to linear stability of the disease-free equilibrium. We show with numerical experiments how the within-host dynamics influences the epidemiological outcomes. 
%
    We present an epidemiological model for vector-borne diseases that includes within-host viral load and antibody dynamics using structured transport equations. By incorporating the internal dynamics into the infected and recovered host compartments, the formulation introduces nonlinearities and nonlocalities. We establish analytical properties, including well-posedness and mass conservation, and characterize its characteristic curves. Furthermore, we derive a simplified Uniform Host Response (UHR) model featuring delay-type terms. For both the full and UHR frameworks, the basic reproduction number $\mathcal{R}_0$ is determined and shown to serve as a threshold for the existence of an endemic equilibrium, and is related to the linear stability of the disease-free state. Finally, numerical experiments, parameterized specifically for Dengue fever, demonstrate how within-host mechanisms influence population-level epidemiological outcomes.
\end{abstract}
%%%%%%%%%%%%%%%%%%%%%%%%%%%%%%%%%%%%%%%%%%%%%%%%%%%%%%%%%%%%%%%%%%
%


%
%
%
%%%%%%%%%%%%%%%%%%%%%%%%%%%%%%%%%%%%%%%%%%%%%%%%%%%%%%%%%%%%%%%%%%
\section{Introduction}
\label{sec:Intro}
%%%%%%%%%%%%%%%%%%%%%%%%%%%%%%%%%%%%%%%%%%%%%%%%%%%%%%%%%%%%%%%%%%
Vector-borne diseases (VBDs) represent a major global health burden, accounting for more than 17\% of all infectious diseases and causing over 700,000 deaths annually. Among these, Dengue fever stands out as a significant public health concern, with approximately half of the global population at risk and an estimated 100 to 400 million infections occurring each year \cite{WHO2024}. Transmitted primarily through the {\em Aedes} species mosquito, Dengue presents unique epidemiological challenges due to the existence of four antigenically distinct serotypes, DENV1--DENV4 \cite{st2019adaptive}. 

While primary infection with one serotype provides life-long homologous immunity, cross-protection against heterologous serotypes is only temporary \cite{weiskopf2014t}. Furthermore, subsequent infections with a different serotype carry a higher risk of severe disease or mortality, a risk largely attributed to the phenomenon of {\em antibody-dependent enhancement (ADE)} \cite{st2019adaptive}. Developing mathematical frameworks to quantify the relevance of ADE is crucial for understanding Dengue dynamics. In this article, we take an initial step toward this goal by proposing a multiscale model that couples ordinary differential equations with partial differential equations, the latter of which are structured by within-host viral load and antibody levels.

% \hrule
% Dengue fever is a vector-borne disease transmitted from person to person through the {\em Aedes} species mosquito. Currently, approximately half of the global population is at risk of contracting dengue, with an estimated 100 to 400 million infections occurring annually \cite{WHO2024}. Consequently, dengue has become a significant public health concern in many regions worldwide. 
% Humans can experience symptomatic dengue virus (DENV) infections more than once due to four antigenically distinct serotypes \cite{st2019adaptive}, namely DENV1-DENV4. 
% Individuals infected with one serotype maintain a life-long protective immunity to infection by the homologous virus, but protective immunity to infection with heterologous serotypes is only temporary \cite{weiskopf2014t}. 
% While primary natural dengue infection is often asymptomatic, a second infection with a different serotype has more chances to be severe or even deadly \cite{st2019adaptive}, mainly related to the phenomenom called {\em antibody-dependent enhancement (ADE)} \cite{st2019adaptive}. Trying to model this phenomenom to better understand and hopefully quantify its relevance, is of main importance. In this article, we take an initial step towards that goal, by first proposing a model combining ordinary and partial differential equations, the latter stratified by viral load and antibody levels.

Mathematical modeling has historically played a pivotal role in understanding dengue transmission dynamics.
Models date back from the 1970s \cite{fischer1970observations}, with time-discrete sequential models. 
Compartmental SIR-type models with vector-host interactions have been the most widely used framework \cite{bailey1975mathematical,dietz1975transmission,esteva1998analysis,andraud2012dynamic,aron1982population,feng1997competitive}. 
These systems often incorporate additional compartments to account for the existence of multiple strains \cite{aguiar2022mathematical,feng1997competitive,ferguson1999effect}, some of them considering possible cross-immunity between serotypes and/or the occurrence of ADE. Age-structured \cite{pongsumpun2003transmission,ferguson1999transmission} and spatial-dependent \cite{cummings2004travelling,takahashi2005mathematical}  models address heterogeneity in transmission, while agent-based models capture fine-scale behavioral and environmental factors \cite{hunter2017taxonomy}. Recent advances integrate climate variables (e.g., temperature-dependent mosquito traits) \cite{chen2012modeling} and control measures (insecticides \cite{carvalho2019mathematical}, vaccination \cite{aguiar2017mathematical}, {\em Wolbachia} \cite{bliman2018ensuring,almeida2024optimal}, sterile insect technique \cite{aronna2020nonlinear,almeida2022optimal}). 
Challenges remain in parameterizing serotype interactions and quantifying the impact of novel interventions, but these models continue to guide public health strategies globally.

\if{See e.g.
\begin{itemize}
    \item Fischer DB, Halstead SB. Observations related to pathogenesis of dengue hemorrhagic fever. V. Examination of age specific sequential infection rates using a mathematical model. The Yale journal of biology and medicine. 1970 Apr;42(5):329.
\end{itemize}
}\fi

One particularly active field of research is that of models containing within-host dynamics, where the microscopic interactions between (say) virus particles and immune cells is explicitly modeled. The challenge is to link the insights from the microscale with the dynamics of the epidemic at the population level, where more traditional models act. There have been many recent works on this subject, and we mention here \cite{doran2023mathematical,mideo2008linking,childs2019linked,feng2012model,almocera2018multiscale} as a small sample. The common feature of many of these models is to effectively link the epidemiological and the immunological processes by means of multiscale systems.

A natural next step is to extend within-host approaches to vector-borne diseases, which can also involve within-vector modeling. In this context, only a few works are available, including \cite{garira2019coupled,agusto2019transmission}
for malaria, \cite{gulbudak2020immuno} for a general vector-borne disease, and \cite{loisel2024within,saldana2025multiscale,nunez2025dynamic} focusing on arboviruses. 

Our model includes antibody and viral load structures for the infected host population, modeled by transport equations (commonly called structured equations in this context) making it a model with within-host dynamics. 
The use of structured models in epidemiology goes back to the very first contributions by Kermack {\it et al.} \cite{kermack,kermack2,kermack3}, and has been an area of active study in recent decades: see, for instance, \cite{diekmann2013mathematical,auger2008structured,angulo,hethcote}. A main reference work where many developments are presented and an extensive bibliography is collected, is the book \cite{perthame}.

As it turns out, age or time since infection structures are the most widely used in the literature \cite{perthame}. However, it is possible to consider other structure variables, either through transport or kinetic-type terms, such as disease awareness and social factors \cite{amorim2024modeling,zanella2023kinetic}, viral dynamics \cite{gulbudak2020immuno,della2022sir,della2023sir} or antibody levels \cite{gulbudak2020infection}. It seems especially promising to model within-host dynamics as structure variables in the context of vector-borne diseases.

It is clear that the ADE phenomenon depends crucially on the within-host dynamics. Indeed, it is thought to originate from detrimental interactions between IgG antibodies remaining from a primary infection and a secondary infection with a different serotype, see \cite{st2019adaptive,clapham2016modelling}. As a first step to model this phenomenon in the context of a structured model, we consider an infected host population structured by antibody level and viral load, in this work restricted to a single serotype and a single (short-lived) IgM antibody.

We propose a host–vector model for vector-borne disease in which the human population is divided into susceptible, exposed, infected, and recovered compartments, while the vector population consists of exposed and infectious vectors. Susceptible and exposed hosts depend only on time, whereas the infected host population is structured by viral load and antibody level, allowing the model to capture the joint evolution of pathogen dynamics and immune response within individuals. The infected density therefore describes the distribution of infected hosts with respect to these two variables, while the recovered population is structured only by antibody level, reflecting the persistence of immune response after viral clearance. Recovered individuals no longer carry detectable viral load and cannot infect vectors, but they remain temporarily protected from reinfection until their antibody levels decline sufficiently to return to the susceptible class. Exposed hosts represent individuals in the incubation phase, with very low viral load and no antibodies, who eventually progress to the infected class and can transmit the virus to vectors. The vector population includes exposed vectors undergoing the extrinsic incubation period and infectious vectors capable of transmitting the disease to human hosts. 
For the proposed model we provide basic properties such as the characterization of the characteristic curves, well-posedness of the full model, derivation of an expression for the {\it basic reproduction number} $\mathcal{R}_0$ and of the disease-free equilibria (DFE). We additionally derive a related simplified model associated to a uniform host response, that allows a simpler analysis and simulation. We present a series of numerical simulations that illustrate the behaviour of our model, and its sensitivity with respect to key parameters. Throughout all the work, and in particular in the numerical simulations, the prototype application of our model is Dengue fever.

% \bigskip

% \SA{XXX: @Paulo, @Debora: Do you think we need further comparison with other results of the kind (look for other articles structures by viral load / antibody level)? Maybe not necessary. Intro long enough, other articles already cited}
%
%
%

% \bigskip

{The manuscript is organized as follows. In Section \ref{sec:Eqns}, we introduce the full model and describe the variables and parameters in detail. Section \ref{sec:properties} is devoted to the presentation of basic analytical properties of the system, including the description of the characteristic curves, mass conservation, well-posedness, and an expression for $\mathcal{R}_0$. In Sections \ref{subsec:DelayModel} and \ref{sec:Dynamics}, we derive and analyze a simplified model with uniform host response. Section \ref{sec:Numerical} presents the numerical simulations, which are followed by a discussion section. Technical proofs are provided in Section~\ref{sec:proofs} and supplementary materials.}

\section{Presentation of the model}
\label{sec:Eqns}
%%%%%%%%%%%%%%%%%%%%%%%%%%%%%%%%%%%%%%%%%%%%%%%%%%%%%%%%%%%%%%%%%%

%
%
%
%%%%%%%%%%%%%%%%%%%%%%%%%%%%%%%%%%%%%%%%%%%%%%%%%%%%%%%%%%%%%%%%%%
\subsection{A structured integro-differential model: outline}
\label{subsec:FullModel}

% \subsubsection{Outline of the model}

We consider a population of human hosts, which we divide into four compartments, namely susceptible, exposed, infected, and recovered. The total host population is assumed to be constant, with value $N_h$. The susceptible and exposed host populations are denoted by $S(t)$ and $E(t),$ respectively, and depend only on the time $t\ge 0$. The infected host population, $I,$ is structured by viral load $z$ and antibody levels $y$. Therefore, $I(t,z,y)$ denotes the density of infected hosts at time $t \geq 0$, with respect to the viral load $z$ and antibody level $y.$

We introduce a minimum detectable viral load $z_0$, and assume that the viral load $z$ varies in $[z_0, + \infty)$ \cite{xu2020defervescent}, while the antibody level $y$ varies in $[0,+ \infty)$. In this way, 
$\int_{z_1}^{z_2} \int_{y_1}^{y_2} I(t,z,y) \, dy dz$ represents
the total infected population at time $t$ having viral load in $[z_1, z_2]$ and antibody level in $[y_1,y_2]$. The recovered population, meanwhile, is structured by the antibody variable $y$ only, and to distinguish it from the susceptible population (which has no significant antibody or virus levels), we suppose that the recovered population has antibody level $y \ge y_0$, for some (small) antibody level $y_0$ which will be determined below. Thus,
the recovered individuals at time $t \geq 0$ are represented by a  density $R(t,y)$ with respect to the antibody level $y$ and $\int_{y_1}^{y_2} R(t,y) \, dy$ represents the total recovered population at time $t$, with antibody level in $[y_1,y_2] \subset [y_0,+\infty)$. The recovered population no longer has an apparent viral load, so they cannot transmit the disease to vectors. However, they remain with antibodies and cannot be reinfected during their recovery period, until the antibody level is so low that they can be considered susceptible again.

Susceptible hosts do not have significant viral load or antibody levels, and so are modeled as a function of time only, $S(t)$. The exposed host population, $E(t)$, has been infected with the dengue virus, but it is still incubating and so there is no viremia. As a result, the viral load is very low, and no antibodies are produced \cite{st2019adaptive}. Over time, exposed hosts reach higher viral load and enter the infected compartment. The infected hosts have a significant viral load (in our model, larger than $z_0$), and a quantifiable antibody level. During this viremia period, infected individuals can infect susceptible mosquitoes. 

The vector population is divided into the exposed vectors $E_v(t)$ (when the virus is in the extrinsic incubation stage) and the infected vectors $I_v(t)$ (when the vectors can transmit the disease to human hosts). We consider a constant total population of vectors $N_v$, and if necessary, the susceptible vectors can be computed by $S_v(t) = N_v - E_v(t) - I_v(t)$.

We now explain in more detail how these modeling assumptions are implemented in our model.

\subsection{Within-host dynamics}
\label{sec:within-host-dynamics}

\subsubsection{Infected dynamics}
Let us first consider the infected compartment $I$, which is structured by viral load $z$ and antibody level $y$. As a structured model, the density $I(t,z,y)$ varies according to
\begin{equation}
  \label{eq:3}
  \frac{\partial }{\partial t} I(t,z,y) + \frac{\partial}{\partial z}\big(V_z \,I(t,z,y)\big) + \frac{\partial }{\partial y} \big(V_y \, I(t,z,y)\big) = 0.
\end{equation}
In such a model, the profile of the density $I(t,z,y)$ in the $(z,y)-$plane moves with velocity $V_y$ in the $y$-direction, and with velocity $V_z$ in the $z$-direction. To define the model, therefore, it is necessary to provide the velocities $V_y,V_z$, which may depend on $t,z,y$, or on other quantities related to the system. 

In infected individuals $I$, we choose the velocities $V_z,V_y$ according to the following phenomenological description of the within-host viral-antibody dynamics \cite{CDCwebsite,st2019adaptive}. We suppose that once the viral load is larger than the minimum detectable value $z_0$, the virus replicates with a rate  $a_1$. This corresponds to a term of the form $a_1 z$ in the velocity term $V_z$, which represents exponential growth of the viral population.
In parallel, the presence of antibodies hinders the virus population, and so $V_z$ should, at least, include a term of the form $- a(y)$, with $a(y)>0$ descibing precisely how the antibody level lowers the growth rate of the virus population. It is reasonable to assume that this effect is monotone increasing in $y$, and so we take $a(y)= a_2 y$ for some fixed $a_2$. Therefore, we set
\begin{equation}
  \label{eq:1}
  V_z = a_1 z - a_2 y,
\end{equation}
while stressing that this is only the simplest reasonable choice\footnote{An example of a more general choice is to take $V_z = a_1 z(1 - z/K_z) - \overline{a}(z,y)$ to describe logistic growth in the virus load (with a carrying capacity $K_z$) and some $\overline{a}(z,y)>0$ governing the (infinitesimal) viral load decrease with respect to $z$ and $y$, or even nonlocal terms as in \cite{amorim2024modeling}. As mentioned, it is reasonable that $\frac{\partial}{\partial y}\overline{a}(z,y) >0$, and the simplest such choice gives $\overline{a}(z,y) = a_2 y$, resulting in \eqref{eq:1}.}.

Similarly, the drift term $V_y$ should account for the growth of the antibody level in response to the presence of virus, and for the natural decay of the antibody level. A simple choice is then
\begin{equation}
  \label{eq:2}
  V_y = -a_3 y + a_4 z.
\end{equation}
Again, a general form $V_y = b(z,y)$ could be considered, with $\frac{\partial}{\partial y} b(0,y) <0$ (reflecting the natural decay of the antibody level) and $\frac{\partial}{\partial z}b(z,y) >0$ (reflecting the fact that virus presence stimulates antibody production).

All in all, \eqref{eq:3}--\eqref{eq:2} give
\begin{equation}
  \label{eq:4}
  \frac{\partial }{\partial t} I(t,z,y) + \frac{\partial}{\partial z}\big( (a_1 z  -a_2 y)\, I(t,z,y) \big) + \frac{\partial }{\partial y} \big((-a_3 y + a_4 z)\, I(t,z,y)\big) = 0
\end{equation}
for the within-host dynamics of the virus and antibody in the infected compartment,
with $z\in [z_0,+\infty)$ and $y\in [0,+\infty)$.

\subsubsection*{The antibody threshold, $y_0$}

According to our assumptions, it is natural that individuals enter the infected compartment $I$ having a low antibody level. After that, the within-host dynamics plays out, where the virus initially replicates, which motivates antibody production. In turn, this allows for the reduction of the viral load, until it reaches the detectable threshold $z_0$ again, but this time in a  decreasing way.

% Mathematically, a necessary condition for this behavior is expressed by the direction of the vector field $(V_z,V_y),$ defined in \eqref{eq:1},\eqref{eq:2}, along the line $Z_0 :=\{ (z_0,y) :y\ge 0\}$ (see Fig.~\ref{fig:IDynamics}): it should hold that $V_z>0$ for low values of $y$, and, for higer values of $y$, that $V_z<0$.

A sufficient condition for this behavior can be formulated in terms of the direction of the vector field $(V_z,V_y),$ defined in \eqref{eq:1},\eqref{eq:2}, along the half-line $\{ (z_0,y) :y\ge 0\}$ (see Fig.~\ref{fig:IDynamics}). Specifically,  $V_z$ must be positive for small values of y and negative for sufficiently large values of y.

The form of the transport terms in \eqref{eq:4} reflects exactly this type of behavior, and in addition allows us to calculate the antibody threshold $y_0$ where the sign change of $V_z$ takes place. From \eqref{eq:1} we see that $V_z >0
$ if, and only if, $y<z_0 a_1/a_2$. Therefore, we define the antibody threshold value $y_0$ as
\begin{equation}
  \label{eq:5}
  \begin{aligned}
    y_0 := z_0\frac{a_1}{a_2}.
  \end{aligned}
\end{equation}
Then, it holds that the population with viral load $z_0$ enters the infected compartment when its antibody level is in $[0,y_0)$, and leaves the infected compartment (into the recovered compartment) when its antibody level is in $(y_0,+\infty)$. How this inflow and outflow take place will be precisely formulated when we discuss the boundary conditions and characteristic curves below, and is illustrated in Fig.~\ref{fig:IDynamics}; in particular, the condition \eqref{eqns:29} on the coefficients $a_{i}$ will be needed. For the moment we observe only that on the outflow boundary $Z_{\mathrm{out}} := \{(z_0,y) : y> y_0\}$ no condition needs to be imposed.

We remark that this can be interpreted as a (necessary) distinction between infected individuals who are in the beginning stages of infection (the viral load is low, but increasing, and the antibody level is low), and in the later stages of infection (where the viral load is low, but decreasing, and the antibody level is higher). This is analogous to the distinct compartments introduced in earlier works (such as \cite{gulbudak2020immuno,della2023sir}) for infected individuals having increasing or decreasing viral loads. Our model, on the other hand, exhibits this distinction as a consequence of our assumptions and mathematical framework. We also remark that the antibody threshold is not an additional parameter of the system, but is determined according to \eqref{eq:5}.

\subsubsection{Recovered dynamics}
\label{sec:recovered-dynamics}

As discussed above, individuals in the infected compartment reach the detectable viral load threshold $z_0$ at the end of the infective period, having an antibody level $y>y_0$; at this point we assume that they enter the recovered compartment $R$. In the recovery phase, we keep track of the antibody level, but not of the viral load. Thus, $R=R(t,y),$ with $y>y_0$. During recovery, the antibody level remains above $y_0$, and we suppose that this prevents reinfection. This is consistent with the description of $y_0$ in the previous discussion. The antibody has a natural decay rate $-a_5$, and when the antibody level of an individual reaches $y_0$ (necessarily from above), they leave the recovered compartment and enter the susceptible compartment. In this way, the recovered population should be a solution of the advection equation
\begin{equation*}
  \begin{aligned}
    \frac{\partial }{\partial t} R(t,y) + \frac{\partial }{\partial y} \big( -a_5 y R(t,y)\big) = F_{R},
  \end{aligned}
\end{equation*}
where the source term $F_{R}$ accounts for the infected individuals at the end of their infective period who reach the viral load threshold $z_0$, and thus enter the recovered compartment.

The quantity $F_{R}$ is computed from the outgoing flux of the infected population along the boundary $Z_{\mathrm{out}} = \{(z_0,y) : y> y_0\}$ as:
\begin{equation}
  \label{eqns:4}
  F_{R} = -(V_z \, I)_{|_{Z_{\mathrm{out}}}} = -I(t,z_0, y) \big( a_1 z_0 - a_2 y \big), 
\end{equation}
giving
\begin{equation}
  \label{eq:6}
  \begin{aligned}
    \frac{\partial }{\partial t} R(t,y) - \frac{\partial }{\partial y} \big(a_5 y R(t,y)\big) = -I(t,z_0, y) \big( a_1 z_0 - a_2 y \big), \qquad y \in [y_0,+\infty),
  \end{aligned}
\end{equation}
for the evolution of the recovered compartment. Note that the right-hand side is nonnegative, from \eqref{eq:5} and $y>y_0$.
As the flux in \eqref{eq:6} is outgoing at $y=y_0$, there is no need to impose a boundary condition there.

\subsection{Susceptible and exposed dynamics}
\label{sec:susc-expos-dynam}

The inflow into the susceptible compartment should correspond to the outgoing flux of the recovered population (with the opposite sign). According to \eqref{eq:6}, the outflow at $y=y_0$ is $-a_5 y_0 R(t,y_0)$. On the other hand, individuals leave the susceptible compartment (and enter the exposed compartment) due to infection, which may occur when there is contact between the susceptible individual and an infective mosquito vector, whose quantity is $I_v(t)$. This is modulated by the \emph{biting rate} $b,$ which has units of (time)$^{-1}$, and the \emph{vector-host infection probability} $\beta_{vh}$. We suppose the vector and host populations are well-mixed, and so infection occurs through a mass-action law. This description yields the equation
\begin{equation}
  \label{eq:7}
  \dot S(t) = - b\beta_{vh} I_v(t) \frac{S(t)}{N_h} + a_5 y_0 R(t,y_0), \qquad t\ge 0
\end{equation}
for the susceptible compartment.

Similarly, the exposed (host) compartment gains the population which leaves the susceptible compartment, and in turn loses individuals to the infected compartment after an incubation period $\tau_h$. Thus, the equation for the exposed compartment is
\begin{equation}
  \label{eq:8}
  \dot E(t) = b \beta_{vh} I_v(t) \frac{S(t)}{N_h} - \frac{1}{\tau_h}E(t), \qquad t\ge 0.
\end{equation}

\subsection{From exposed to infected host}
\label{sec:from-expos-infect}

To describe the host dynamics, it remains only to show how exposed hosts $E(t)$ enter into the infected host compartment $I(t,z,y)$. From \eqref{eq:8}, exposed hosts leave the exposed compartment according to the term with $\frac{1}{\tau_h}E(t)$, and so this quantity should flow into the infected compartment. As mentioned in the discussion in Section~\ref{sec:within-host-dynamics}, the introduction of newly infective individuals into the $I$ compartment is represented by an inflow term into the domain $\{ (z,y) : z>z_0, y>0\}$, through the inflow boundary $Z_{\mathrm{in}}:= \{(z_0,y) : 0\le y\le y_0\}.$ This amounts to setting a boundary condition for the equation \eqref{eq:4} along $Z_{\mathrm{in}}$.

To characterize this boundary condition, we define a probability function $g:[0,y_0] \to \RR$, governing the antibody level distribution of the inflowing infective population. Then, we set $I(t,z=z_0,y\in[0,y_0]) = \frac{C}{\tau_h}E(t)  g(y)$, where the scaling factor $C$ is chosen to ensure that the total mass of the system is conserved. It turns out (as we will see in Prop.~\ref{prop:cons} below) that $C$ must be $(a_1 z_0 - a_2 y^\star)^{-1}$, where
\begin{equation}
  \label{eqns:y_star}
  y^\star = \int_{0}^{y_0} yg(y) \,dy \in (0,y_0)
\end{equation}
(thus, $y^\star$ is the expected value of the probability function $g$). Therefore, we set
\begin{equation}
  \label{eq:9}
  I(t,z=z_0, y \in [0,y_0] ) = \frac{ E(t) g(y) }{(a_1 z_0 - a_2 y^\star)\tau_h}
\end{equation}
as the boundary condition on the inflow boundary $Z_{\mathrm{in}}$ of $I(t,z,y)$. The boundary condition on $\{(z>z_0,y=0)\}$ is set to zero, and as mentioned in Section~\ref{sec:within-host-dynamics}, the boundary $\{ (z_0, y> y_0)\}$ is always an outflow boundary, and as such does not require a boundary condition.

\subsection{Dynamics of the vector population}
\label{sec:dynamics-vector}

We will model the vector population by directly considering only their exposed and infected compartments, respectively $E_v(t)$ and $I_v(t)$. Indeed, we suppose that the susceptible vector population, $S_v(t)$, is replenished in such a way that the total vector population remains a constant $N_v$, even in the presence of vector mortality. This allows us to write $S_v = N_v-E_v-I_v$, and so we should have
\begin{equation*}
  \left\{
    \begin{aligned}
      & \dot E_v(t) = b \frac{I_h(t)}{N_h} S_v(t) -\frac{1}{\tau_v}E_v(t) - {\mu_v} E_v(t),
      \\
      & \dot I_v(t) = \frac{1}{\tau_v} E_v(t) - {\mu_v} I_v(t),
    \end{aligned}\right.
\end{equation*}
where $b$ is again the biting rate, and $I_h(t)$ represents the total quantity of infected hosts, multiplied by the probability of host-vector transmission. More precisely,
we can use $I(t,z,y)$ from the equation \eqref{eq:4} to express $I_h(t)$ as
\begin{equation*}
  I_h(t) := \int_{0}^{\infty}\!\!\int_{z_0}^{\infty} I(t,z,y) \beta_{hv}(z) \,dz dy,
\end{equation*}
where the function $\beta_{hv}(z)$ is the \emph{host-vector infection probability}, which naturally depends on the viral load $z$.

The vector dynamics are then given by
\begin{equation}
  \label{eq:10}
  \left\{
    \begin{aligned}
      & \dot E_v(t) = \frac{b}{N_h} \int_{0}^{\infty}\!\!\int_{z_0}^{\infty} I(t,z,y) \beta_{hv}(z) \,dz dy  \,\big( N_v - E_v(t) - I_v(t)\big) -\frac{1}{\tau_v}E_v(t) - {\mu_v} E_v(t),
      \\
      & \dot I_v(t) = \frac{1}{\tau_v} E_v(t) - {\mu_v} I_v(t), \qquad t\ge 0.
    \end{aligned}\right.
\end{equation}

\subsection{Full model}
\label{sec:full-model}

Our full model consists of the equations \eqref{eq:4},\eqref{eq:6},\eqref{eq:7},\eqref{eq:9} and \eqref{eq:10}. Since the total host and vector populations are constant, we will consider the proportions
\begin{equation*}
  \overline S = \frac{S}{N_h}, \quad \overline E = \frac{E}{N_h}, \quad \overline I = \frac{I}{N_h}, \quad \overline R = \frac{R}{N_h},
% \end{equation*}
% \begin{equation*}
  \overline E_v = \frac{E_v}{N_v}, \quad \overline I_v = \frac{I_v}{N_v}, 
\end{equation*}
and drop the bars from the notation. This way, setting
\begin{equation*}
  m:= \frac{N_v}{N_h},
\end{equation*}
we obtain the system
\begin{equation}\label{eqns:FullModel}
  \left\{
    \begin{aligned} 
      & \frac{\partial I}{\partial t} (t,z,y) + \frac{\partial }{\partial z} \big(( a_1 z - a_2 y) I(t,z,y)  \big)
       + \frac{\partial }{\partial y} \big( (-a_3 y  + a_4 z ) I(t,z,y)  \big) = 0, \qquad z \ge z_0, y\ge 0,
      \\
      & \frac{\partial R}{\partial t} (t,y) - \frac{\partial }{\partial y} \left( a_5 y R(t,y) \right)  = - I(t,z_0, y) \bigl( a_1 z_0 - a_2 y \bigr), \qquad y \ge y_0,
      \\
      & \dot{S}(t) = -b\beta_{vh} m I_v(t) {S(t)} + a_5y_0 R(t,y_0),
      \\
      & \dot{E}(t) = b \beta_{vh} m I_v(t) {S(t)} - \frac{1}{\tau_h}E(t),
      \\
      & \dot{E}_v(t) =  b \int_0^{\infty}\!\!\int_{z_0}^{\infty} {I(t,z,y)} \beta_{hv} (z) \,dz dy   \,\big( 1 - E_v(t) - I_v(t) \big)
       - \Big(\frac{1}{\tau_v} + \mu_v\Big) E_v(t), 
      \\
      & \dot{I}_v(t) = \frac{1}{\tau_v} E_v(t)- \mu_v I_v(t),
    \end{aligned}\right.
\end{equation}
along with the boundary conditions
\begin{equation}
  \label{eqns:BCI}
  \left\{
  \begin{aligned}
    &I(t,z=z_0, y \in [0,y_0] ) = \frac{ E(t) g(y) }{(a_1 z_0 - a_2 y^\star)\tau_h} ,
    \\
    % &I(t,z \rightarrow + \infty, y) = 0 ,
    % \\
    % &I(t,z , y \rightarrow + \infty) = 0 ,
    % \\
    &I(t,z >z_0, y = 0) = 0
  \end{aligned}\right.
\end{equation}
where $y^\star$ is given by \eqref{eqns:y_star}, and the initial conditions (all nonnegative)
\begin{equation}
  \label{eqns:ini}
  \left\{
    \begin{aligned}
      &I(0,z,y) = I_0(z,y), \, z\ge z_0, y\ge 0, \quad R(0,y) = {R}_0(y),\, y\ge y_0,
      \\
      & S(0) = S_0, \quad E(0)=E_0, \quad E_v(0) = E_{v0}, \quad I_v(0) = I_{v0}.
    \end{aligned}\right.
\end{equation}
Since our model quantities are proportions of the total population, we assume moreover that
\begin{equation}
  \label{eqns:28}
  \begin{aligned}
    &S_0 + E_0 + \int_{z_0}^{\infty}\!\!\int_{0}^{\infty} I_0(z,y) \,dydz + \int_{y_0}^{\infty} R_0(y) \,dy = 1,
    \\
    &I_{v0} + E_{v0} \le 1.
  \end{aligned}
\end{equation}

We also suppose that $R_0(y)$ and $I_0(z,y)$ are integrable on their domains and that $R(t, y \rightarrow + \infty) = 0$, $I(t,z,y) \to 0$ as $z^2+y^2 \to \infty$.

% \begin{equation}
%   \label{eqns:BCR}
%   R(t, y \rightarrow + \infty) = 0 .
% \end{equation}

%

The variables and parameters present in our model are specified in Table \ref{tab:ModelParameters}. %\tau_v$ and $\tau_h$ are the extrinsic (within the vector) and intrinsic (within the host) incubation period of the pathogen, respectively. %These periods are approximately 7 days for $\tau_h$ \cite{CDCwebsite,xu2020defervescent}. 

\begin{table}[htbp]
  \small
  \centering
  \caption{Variables and parameters in the system \eqref{eqns:FullModel}.}
  \begin{tabular}{lllll}
    \hline
    Notation & Units & Definition & References  \\ 
    \hline  
    $t$  & $t$ &  Time  & -  \\
    $z$ & $z$ &  Viral load & \cite{laue1999detection,nguyen2013host,xu2020defervescent} \\
    $y$ & $y$  &   Antibody level & \cite{assir2014concurrent}\\[2pt] \hline
    Variables &  & & \\ \hline
    $I(t,z,y)$ & -- & Infected host density w.r.t.~viral load and antibody titer & \\
    $R(t,y)$ & -- & Recovered host density w.r.t.~antibody titer & \\
    $S(t)$ & -- & Susceptible host proportion & \\
    $E(t)$ & -- & Exposed host proportion & \\
    $E_v(t)$ & -- & Exposed vector proportion & \\
    $I_v(t)$ & -- & Infected vector proportion & \\ \hline
    Parameters  &  &  & \\ \hline
    $N_h,N_v$ & number & Population number of hosts and vectors & \\
    $m$ & -- & $N_v/N_h$ & \\
    $z_0$ & $z$ &  Minimum detectable viral load ($z_0 > 0$) & \cite{xu2020defervescent}\\
    $y_0$ & $y$ &  Antibody threshold value, $y_0 =z_0 a_1/a_2$ & \eqref{eq:5} \\
    $a_1$ & $t^{-1}$ &  Virus growth rate in host & -- \\
    $a_2$ & $z(ty)^{-1}$ &  Viral load decay rate per unit of antibody  & Fig.~S6, \cite{nguyen2013host} \\ %0.60
    $a_3$ & $t^{-1}$ &  Antibody decay rate on infected host & \cite{sebayang2021modeling}  \\
    $a_4$ & $y(tz)^{-1}$ &  Antibody production rate per virus load unit  &  \cite{sebayang2021modeling} \\
    $a_5$  & $t^{-1}$ &  Antibody titer decay rate on recovered host & --  \\
    $b$  & $t^{-1}$  & Mosquito biting rate  & \cite{aguiar2022mathematical} \\
    $\beta_{vh}$ & -- & Vector to host transmission probability & \\ % 
    $\mu_v$ & $t^{-1}$ &  Natural vector death rate & \cite{massad2010estimation} \\
    $\beta_{hv}(z)$ & -- &  Host to vector transmission probability &  \cite{nguyen2013host} \\
    $g(y)$ & -- &  Exposed host to infected host distribution function &  -- \\
    $y^\star$ & $y$ & Mean antibody titer of recently infected hosts & -- \\
    % probability distribution function
    $\tau_v$ & $t$ &   Virus incubation period in the vector & \cite{nguyen2013host,CDCwebsite} \\ %\href{https://www.cdc.gov/dengue/training/cme/ccm/page45915.html}{(cf. CDC)}  \\
    $\tau_h$ & $t$ &   Virus incubation period in the host & \cite{massad2010estimation,CDCwebsite,xu2020defervescent} \\ %\href{https://www.cdc.gov/dengue/training/cme/ccm/page45915.html}{(cf. CDC)}  \\
    \hline
  \end{tabular}
  \label{tab:ModelParameters} 
\end{table}
The interactions between the host and vector compartments during the dengue infection process, as described in our model, are illustrated schematically in Fig.~\ref{fig:CompartmentalModelDiagram_with_Parameters}.

% apenas taxas no diagrama
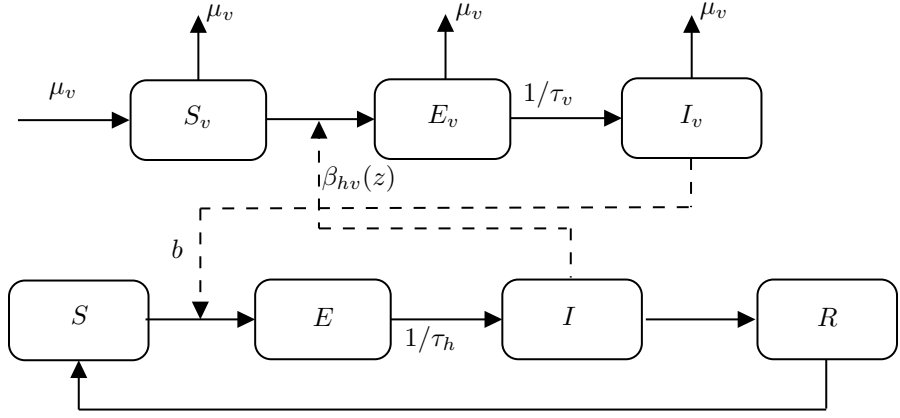
\begin{figure}[ht]
  \centering
  
  \tikzset{every picture/.style={line width=0.75pt}} %set default line width to 0.75pt        

  \begin{tikzpicture}[x=0.75pt,y=0.75pt,yscale=-1,xscale=1]
    % uncomment if required: \path (0,300); %set diagram left start at 0, and has height of 300

    % Rounded Rect [id:dp694505201454906] 
    \draw   (315,63.5) .. controls (315,59.08) and (318.58,55.5) .. (323,55.5) -- (375,55.5) .. controls (379.42,55.5) and (383,59.08) .. (383,63.5) -- (383,87.5) .. controls (383,91.92) and (379.42,95.5) .. (375,95.5) -- (323,95.5) .. controls (318.58,95.5) and (315,91.92) .. (315,87.5) -- cycle ;
    % Rounded Rect [id:dp1251153923817956] 
    \draw   (439,63) .. controls (439,58.58) and (442.58,55) .. (447,55) -- (501,55) .. controls (505.42,55) and (509,58.58) .. (509,63) -- (509,87) .. controls (509,91.42) and (505.42,95) .. (501,95) -- (447,95) .. controls (442.58,95) and (439,91.42) .. (439,87) -- cycle ;
    % Straight Lines [id:da06515270535790196] 
    \draw    (383,75) -- (436,75) ;
    \draw [shift={(439,75)}, rotate = 180] [fill={rgb, 255:red, 0; green, 0; blue, 0 }  ][line width=0.08]  [draw opacity=0] (8.93,-4.29) -- (0,0) -- (8.93,4.29) -- cycle    ;
    % Rounded Rect [id:dp08076574095934252] 
    \draw   (255,164) .. controls (255,159.58) and (258.58,156) .. (263,156) -- (315,156) .. controls (319.42,156) and (323,159.58) .. (323,164) -- (323,188) .. controls (323,192.42) and (319.42,196) .. (315,196) -- (263,196) .. controls (258.58,196) and (255,192.42) .. (255,188) -- cycle ;
    % Rounded Rect [id:dp8093021862126111] 
    \draw   (379,164) .. controls (379,159.58) and (382.58,156) .. (387,156) -- (441,156) .. controls (445.42,156) and (449,159.58) .. (449,164) -- (449,188) .. controls (449,192.42) and (445.42,196) .. (441,196) -- (387,196) .. controls (382.58,196) and (379,192.42) .. (379,188) -- cycle ;
    % Straight Lines [id:da5495366264030856] 
    \draw    (323,176) -- (376,176) ;
    \draw [shift={(379,176)}, rotate = 180] [fill={rgb, 255:red, 0; green, 0; blue, 0 }  ][line width=0.08]  [draw opacity=0] (8.93,-4.29) -- (0,0) -- (8.93,4.29) -- cycle    ;
    % Rounded Rect [id:dp5570221902328516] 
    \draw   (131.5,164.5) .. controls (131.5,160.08) and (135.08,156.5) .. (139.5,156.5) -- (193.5,156.5) .. controls (197.92,156.5) and (201.5,160.08) .. (201.5,164.5) -- (201.5,188.5) .. controls (201.5,192.92) and (197.92,196.5) .. (193.5,196.5) -- (139.5,196.5) .. controls (135.08,196.5) and (131.5,192.92) .. (131.5,188.5) -- cycle ;
    % Straight Lines [id:da4500399162111457] 
    \draw    (200,176) -- (252,176) ;
    \draw [shift={(255,176)}, rotate = 180] [fill={rgb, 255:red, 0; green, 0; blue, 0 }  ][line width=0.08]  [draw opacity=0] (8.93,-4.29) -- (0,0) -- (8.93,4.29) -- cycle    ;
    % Straight Lines [id:da6594412724908458] 
    \draw    (451,176) -- (503,176) ;
    \draw [shift={(506,176)}, rotate = 180] [fill={rgb, 255:red, 0; green, 0; blue, 0 }  ][line width=0.08]  [draw opacity=0] (8.93,-4.29) -- (0,0) -- (8.93,4.29) -- cycle    ;
    % Rounded Rect [id:dp6982448590655768] 
    \draw   (507,164) .. controls (507,159.58) and (510.58,156) .. (515,156) -- (569,156) .. controls (573.42,156) and (577,159.58) .. (577,164) -- (577,188) .. controls (577,192.42) and (573.42,196) .. (569,196) -- (515,196) .. controls (510.58,196) and (507,192.42) .. (507,188) -- cycle ;
    % Straight Lines [id:da7661501011815295] 
    \draw    (541.5,221.25) -- (166.5,221.25) ;
    % Straight Lines [id:da9071522498731999] 
    \draw    (541.5,196) -- (541.5,220.75) ;
    % Straight Lines [id:da4506895986728099] 
    \draw    (166.5,199.5) -- (166.5,221.25) ;
    \draw [shift={(166.5,196.5)}, rotate = 90] [fill={rgb, 255:red, 0; green, 0; blue, 0 }  ][line width=0.08]  [draw opacity=0] (8.93,-4.29) -- (0,0) -- (8.93,4.29) -- cycle    ;
    % Straight Lines [id:da02816279066363614] 
    \draw  [dash pattern={on 4.5pt off 4.5pt}]  (287,130.33) -- (411.33,130.33) ;
    % Straight Lines [id:da32627405590621006] 
    \draw  [dash pattern={on 4.5pt off 4.5pt}]  (472.5,119.67) -- (227.5,120.33) ;
    % Rounded Rect [id:dp849247521654406] 
    \draw   (192.5,64) .. controls (192.5,59.58) and (196.08,56) .. (200.5,56) -- (252.5,56) .. controls (256.92,56) and (260.5,59.58) .. (260.5,64) -- (260.5,88) .. controls (260.5,92.42) and (256.92,96) .. (252.5,96) -- (200.5,96) .. controls (196.08,96) and (192.5,92.42) .. (192.5,88) -- cycle ;
    % Straight Lines [id:da3240747265957795] 
    \draw    (260.5,75.5) -- (312,75.5) ;
    \draw [shift={(315,75.5)}, rotate = 180] [fill={rgb, 255:red, 0; green, 0; blue, 0 }  ][line width=0.08]  [draw opacity=0] (8.93,-4.29) -- (0,0) -- (8.93,4.29) -- cycle    ;
    % Straight Lines [id:da02235812607932819] 
    \draw    (136,76) -- (189,76) ;
    \draw [shift={(192,76)}, rotate = 180] [fill={rgb, 255:red, 0; green, 0; blue, 0 }  ][line width=0.08]  [draw opacity=0] (8.93,-4.29) -- (0,0) -- (8.93,4.29) -- cycle    ;
    % Straight Lines [id:da35627723321359683] 
    \draw    (226.5,56) -- (226.5,26.25) ;
    \draw [shift={(226.5,23.25)}, rotate = 90] [fill={rgb, 255:red, 0; green, 0; blue, 0 }  ][line width=0.08]  [draw opacity=0] (8.93,-4.29) -- (0,0) -- (8.93,4.29) -- cycle    ;
    % Straight Lines [id:da6165810540948216] 
    \draw    (350.5,55.5) -- (350.5,25.75) ;
    \draw [shift={(350.5,22.75)}, rotate = 90] [fill={rgb, 255:red, 0; green, 0; blue, 0 }  ][line width=0.08]  [draw opacity=0] (8.93,-4.29) -- (0,0) -- (8.93,4.29) -- cycle    ;
    % Straight Lines [id:da19424342705532016] 
    \draw    (474,55) -- (474,26.25) ;
    \draw [shift={(474,23.25)}, rotate = 90] [fill={rgb, 255:red, 0; green, 0; blue, 0 }  ][line width=0.08]  [draw opacity=0] (8.93,-4.29) -- (0,0) -- (8.93,4.29) -- cycle    ;
    % Straight Lines [id:da5634526405297302] 
    \draw  [dash pattern={on 4.5pt off 4.5pt}]  (473.67,95) -- (473.67,119.67) ;
    % Straight Lines [id:da6397037433110488] 
    \draw  [dash pattern={on 4.5pt off 4.5pt}]  (227.5,120.33) -- (227.5,173) ;
    \draw [shift={(227.5,176)}, rotate = 270] [fill={rgb, 255:red, 0; green, 0; blue, 0 }  ][line width=0.08]  [draw opacity=0] (8.93,-4.29) -- (0,0) -- (8.93,4.29) -- cycle    ;
    % Straight Lines [id:da26154917328603] 
    \draw  [dash pattern={on 4.5pt off 4.5pt}]  (413.33,130.33) -- (413.33,155) ;
    % Straight Lines [id:da15860988946666454] 
    \draw  [dash pattern={on 4.5pt off 4.5pt}]  (287,79.33) -- (287,130.33) ;
    \draw [shift={(287,76.33)}, rotate = 90] [fill={rgb, 255:red, 0; green, 0; blue, 0 }  ][line width=0.08]  [draw opacity=0] (8.93,-4.29) -- (0,0) -- (8.93,4.29) -- cycle    ;

    % Text Node
    \draw (340,66.9) node [anchor=north west][inner sep=0.75pt]    {$E_{v}$};
    % Text Node
    \draw (466.5,66.9) node [anchor=north west][inner sep=0.75pt]    {$I_{v}$};
    % Text Node
    \draw (161,167.9) node [anchor=north west][inner sep=0.75pt]    {$S$};
    % Text Node
    \draw (282.5,168.9) node [anchor=north west][inner sep=0.75pt]    {$E$};
    % Text Node
    \draw (407.5,168.4) node [anchor=north west][inner sep=0.75pt]    {$I$};
    % Text Node
    \draw (535.5,168.9) node [anchor=north west][inner sep=0.75pt]    {$R$};
    % Text Node
    \draw (217.5,67.4) node [anchor=north west][inner sep=0.75pt]    {$S_{v}$};
    % Text Node
    \draw (150,55.4) node [anchor=north west][inner sep=0.75pt]    {$\mu _{v}$};
    % Text Node
    \draw (229,16.4) node [anchor=north west][inner sep=0.75pt]    {$\mu _{v}$};
    % Text Node
    \draw (354,15.9) node [anchor=north west][inner sep=0.75pt]    {$\mu _{v}$};
    % Text Node
    \draw (476.5,15.4) node [anchor=north west][inner sep=0.75pt]    {$\mu _{v}$};
    % Text Node
    \draw (211.5,134.57) node [anchor=north west][inner sep=0.75pt]    {$b$};
    % Text Node
    \draw (328.5,177.9) node [anchor=north west][inner sep=0.75pt]    {$1/\tau _{h}$};
    % Text Node
    \draw (388,54.9) node [anchor=north west][inner sep=0.75pt]    {$1/\tau _{v}$};
    % Text Node
    \draw (287,97) node [anchor=north west][inner sep=0.75pt]    {$\beta_{hv} (z)$};

  \end{tikzpicture}

  \caption{The schematic diagram of our host-vector model. The parameters are summarized in Table~\ref{tab:ModelParameters}.}
  \label{fig:CompartmentalModelDiagram_with_Parameters}
\end{figure}

\section{Basic properties}
\label{sec:properties}

\subsection{Characteristic curves}
\label{subsec:Characteristic}

The characteristic curves of the transport equation of the infected hosts \eqref{eqns:FullModel}
describe the trajectories of individuals who enter the infected compartment through the inflow boundary, $Z_{\mathrm{in}}= \{(z_0,y) : 0 \le y < y_0\},$ according to the probability distribution $g(y)$, $y\in[0,y_0)$ (see \eqref{eqns:BCI}). These characteristics are the solutions to the following (linear) system of ODEs,
\begin{equation}\label{eqns:ODE}
  \begin{cases}
    \dot{z}= {a_1} z - {a_2} y \\
    \dot{y} = {a_4} z - {a_3} y, 
  \end{cases}
\end{equation}
for $z(0) =z_0$, $y(0) = y' \in [0,y_0)$.

Note that on $Z_{\mathrm{in}}$, we have $\dot z >0$, while on the outflow boundary $ Z_{\mathrm{out}}$ we have $\dot z<0$.% $\{(z=z_0, y>y_0)\}$ we have $\dot z<0$. 
The eigenvalues of the matrix associated to system \eqref{eqns:ODE} are
\begin{equation}
  \label{eqns:eigenvalues}
  \begin{aligned}
    \lambda_{1,2} = \frac{1}{2}\Big(a_1-a_3 \pm \sqrt{(a_1+a_3)^2 - 4a_2a_4}\Big).
  \end{aligned}
\end{equation}
We are interested in the situation where the within-host dynamics are described by an antibody response to the increase in viral load. As described earlier, we expect that the antibody titer should increase while the viral load also initially increases, followed by a decrease in the viral load \cite{CDCwebsite,st2019adaptive}. To achieve this cyclic behavior, it is natural to expect that the eigenvalues of the system \eqref{eqns:ODE} have nonzero imaginary part. Thus, we will suppose throughout the paper that the following assumption holds:
\begin{equation}
  \label{eqns:29}
  4a_2 a_4 > (a_1 + a_3)^2.
\end{equation}
The assumption \eqref{eqns:29} implies, in particular, that the matrix of the system \eqref{eqns:ODE} is nonsingular.

Under the assumption \eqref{eqns:29}, solving \eqref{eqns:ODE}, a standard computation gives the characteristic curves originating from any point $(z_0,y')$, with $y'\in[0,y_0)$:
\begin{equation} 
    \left\{\begin{aligned}
        &z(\tau,y') = e^{\frac{a_1- a_3}{2}\tau} \Big( z_0 \cos{\beta \tau} + \Big(z_0 \frac{a_1+a_3}{2 \beta} - \frac{{y}'  a_2}{\beta }\Big) \sin{\beta \tau} \Big),
        \\
        &y(\tau,y') =  e^{\frac{a_1- a_3}{2}\tau} \Big( {y}' \cos{\beta \tau} +\Big( \frac{ a_4 z_0}{\beta} - {y}' \frac{a_1+ a_3}{2 \beta}\Big)\sin{\beta \tau} \Big),
    \end{aligned}\right.
    \label{eqns:characteristic}
\end{equation}
where
\begin{equation}
  \label{eqns:beta}
  \beta = \frac{1}{2}\sqrt{4 a_2 a_4 - (a_1+a_3)^2 } \in \mathbb{R}.
\end{equation}

Let us define the set $A,$ which will be the domain in $(z,y)$ of the infective compartment $I(t,z,y)$ for each $t$:
\begin{equation}\label{def:A}
\begin{aligned}
    &\text{$A$ is the open set
 bounded by the characteristic curve originating}\\
 &\text{from $(z_0,0)$, and the line $\{z=z_0\}$.}
\end{aligned}
\end{equation}
See Fig.~\ref{fig:Icharacterists}. For each $(z,y)\in A$, there exists a unique pair $(\tau,y') = (\tau(z,y),y'(z,y))$, with $\tau>0$ and $y'\in[0,y_0)$, such that the characteristic originating from $(z_0,y')$ reaches $(z,y)$ after time $\tau$.

For our purposes, we see the solutions of \eqref{eqns:characteristic} as restricted to the domain $A$. A typical example is represented in Fig.~\ref{fig:Icharacterists}.

\begin{figure}[ht!]
    \centering
    \includegraphics[width=0.5\linewidth]{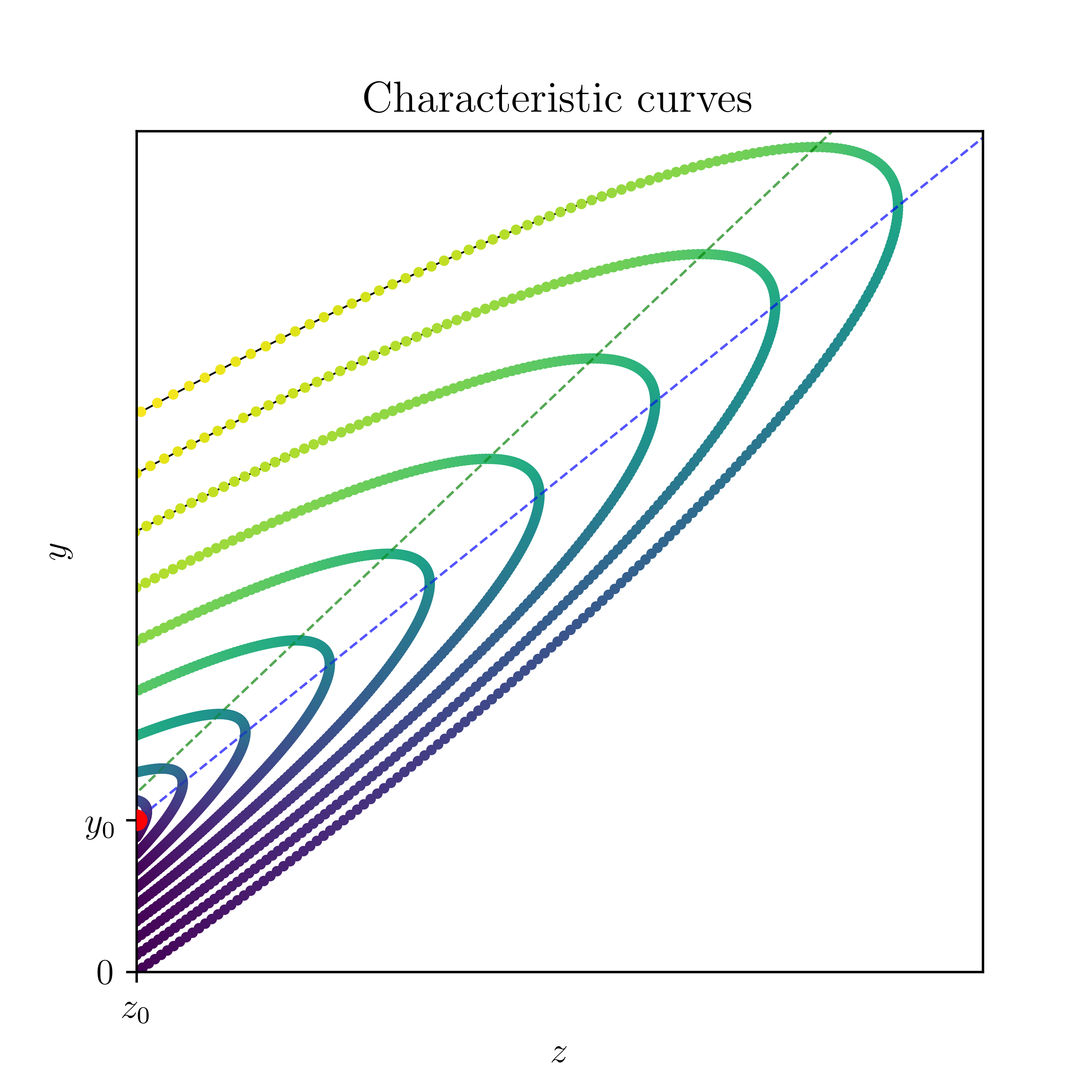}
    \caption{Typical phase portrait for the characteristic curves \eqref{eqns:ODE} of the transport equation for the infected hosts $I$, in the domain $[z_0,\infty) \times [0,\infty)$. The coloring indicates the time evolution along each characteristic. The dashed lines are the nullclines of the system.}
    \label{fig:Icharacterists}
\end{figure}

\subsection{Infective permanence time}

Infected host individuals who enter the $I$ compartment at some time $t$, with an antibody level $y' \in[0,y_0)$, will leave the $I$ compartment into the $R$ (recovered) compartment when some time $\tau_1 = \tau_1(y')$ has elapsed. This \emph{permanence time} in the $I$ compartment is the time elapsed along the characteristic curve \eqref{eqns:characteristic} originating from the inflow boundary $Z_{\mathrm{in}}$ until it reaches the outflow boundary $Z_{\mathrm{out}}.$ 
So, $\tau_1$ is defined as the smallest $\tau \in (0, +\infty)$, such that the characteristic curve $(z(\tau,y'),y(\tau,y'))$ originating from $(z_0,y')$ reaches a point $(z_0,y^+)$ with $y^+>y_0$, see Fig.~\ref{fig:IDynamics} for an illustration. Then, $\tau_1$ solves the system (see \eqref{eqns:characteristic})
\begin{equation} \label{eqns:tau1Characteristics}
    \left\{\begin{aligned}
        &z_0 = e^{\frac{a_1- a_3}{2}\tau_1} \Big( z_0 \cos{\beta \tau_1} + \Big(z_0 \frac{a_1+ a_3}{2 \beta} - \frac{{y}'  a_2}{\beta }\Big) \sin{\beta \tau_1} \Big),
        \\
        & y^+ =  e^{\frac{a_1- a_3}{2}\tau_1} \Big( {y}' \cos{\beta \tau_1} +\Big( \frac{ a_4 z_0}{\beta} - {y}' \frac{a_1+ a_3}{2 \beta}\Big)\sin{\beta \tau_1} \Big).
    \end{aligned}\right.
\end{equation}
More precisely, we have the following result:
\begin{proposition}\label{prop:y}
  Let $z_0>0$ be given, let $y_0=z_0a_1/a_2$, and let $\beta$ be defined by \eqref{eqns:beta}. For each $y'\in[0,y_0)$, there exist a unique $y^+ = y^+(y') > y_0$ and a minimal $\tau_1 = \tau_1(y')>0$ such that \eqref{eqns:tau1Characteristics} holds. $\tau_1$ is the time that an infected individual with antibody level $y'$ at the onset of infection remains in the infected compartment before transitioning to the recovered compartment. Additionally, $\tau_1$ does not depend on the minimum detectable viral load $z_0$.
\end{proposition}
\begin{proof} 

Consider the region $B = \{ (z, y) \in \mathbb{R}^2 :  z>z_0,  y >0 \}$. From the characteristic equations \eqref{eqns:characteristic}
we see that the characteristics enter $B$  only along $\{( z, y ): z=z_0 , y \in [0,y_0)\} \cup \{( z, y) :  z \geq z_0 , y = 0 \},$ and may leave $B$ only along $\{( z, y) : z=z_0 , y \geq y_0\}.$ Therefore, since $y'$ is in the inflowing part of $\partial B$ and the characteristic curve is smooth, the result will be proved if we can show that the characteristic \eqref{eqns:characteristic} issuing from $(z_0,y'),$ $y'\in (0,y_0),$ eventually leaves $B$, since in that case it must leave at some point $(z_0,y^+>y_0),$ and, by continuity, after some minimal time $\tau_1>0$. But for $t = \frac{\pi}{\beta}$, we immediately get from \eqref{eqns:characteristic}, and the assumption that $\beta$ is real, that
\begin{equation}
    \left( z({\pi}/{\beta}), y\left({\pi}/{\beta}\right) \right) = e^{\frac{a_1 -a_3}{2}\frac{\pi}{\beta}} (-z_0, -y'),
\end{equation}
which is in the third quadrant, and therefore outside of $B$. Thus, the characteristic curve in question eventually leaves $B$, which determines $y^+$ and $\tau_1$.

We now check 
 that $\tau_1$ does not depend on $z_0$; this follows from the expression \eqref{eqns:tau1Characteristics}. Indeed, only the first equation of \eqref{eqns:tau1Characteristics} is sufficient to determine $\tau_1$, and writing $y'$ as $\alpha y_0 \equiv \alpha z_0 \frac{a_1}{a_2}$ for some $\alpha\in (0,1)$ allows us to write the first equation of \eqref{eqns:tau1Characteristics} equivalently as
    \begin{equation}\label{eqns:taualt}
    1= e^{\frac{a_1- a_3}{2}\tau_1} \Big( \cos{\beta \tau_1} +  \frac{a_3 - \alpha a_1}{2 \beta}  \sin{\beta \tau_1} \Big).
    \end{equation}
    This shows that $\tau_1$ does not depend on the minimum detectable viral load $z_0$.
\end{proof}

\begin{remark}\label{rmk:tau}
     The equation \eqref{eqns:taualt} and the definition of $\beta$, \eqref{eqns:beta}, show that the $\tau_1$ only depends on $a_2,a_4$ through the product $a_2a_4$.
\end{remark}

\begin{figure}[ht]
	\centering

    \tikzset{every picture/.style={line width=0.75pt}} %set default line width to 0.75pt        
    
    \begin{tikzpicture}[x=0.75pt,y=0.75pt,yscale=-0.9,xscale=0.9]
    %uncomment if require: \path (0,300); %set diagram left start at 0, and has height of 300
    
    %Straight Lines [id:da32946670997185645] 
    \draw [color={rgb, 255:red, 2; green, 73; blue, 155 }  ,draw opacity=1 ] [dash pattern={on 4.5pt off 4.5pt}]  (134,180) -- (189.5,180) ;
    %Straight Lines [id:da8393192166349704] 
    \draw [color={rgb, 255:red, 36; green, 114; blue, 18 }  ,draw opacity=1 ] [dash pattern={on 4.5pt off 4.5pt}]  (190,21) -- (190,264) ;
    %Straight Lines [id:da7515571429034656] 
    \draw    (112.2,263.5) -- (348,263.5) ;
    \draw [shift={(350,263.5)}, rotate = 180] [color={rgb, 255:red, 0; green, 0; blue, 0 }  ][line width=0.75]    (10.93,-3.29) .. controls (6.95,-1.4) and (3.31,-0.3) .. (0,0) .. controls (3.31,0.3) and (6.95,1.4) .. (10.93,3.29)   ;
    %Straight Lines [id:da0950061152285161] 
    \draw    (132.5,279.63) -- (132.5,14.12) ;
    \draw [shift={(132.5,12.13)}, rotate = 90] [color={rgb, 255:red, 0; green, 0; blue, 0 }  ][line width=0.75]    (10.93,-3.29) .. controls (6.95,-1.4) and (3.31,-0.3) .. (0,0) .. controls (3.31,0.3) and (6.95,1.4) .. (10.93,3.29)   ;
    %Curve Lines [id:da11085021389239125] 
    \draw [color={rgb, 255:red, 2; green, 73; blue, 155 }  ,draw opacity=1 ]   (189.67,222.17) .. controls (311.39,185.68) and (279.99,30.4) .. (191.01,69.56) ;
    \draw [shift={(189.67,70.17)}, rotate = 335.33] [color={rgb, 255:red, 2; green, 73; blue, 155 }  ,draw opacity=1 ][line width=1.5]    (10.93,-3.29) .. controls (6.95,-1.4) and (3.31,-0.3) .. (0,0) .. controls (3.31,0.3) and (6.95,1.4) .. (10.93,3.29)   ;
    
    % Text Node
    \draw (203,276) node  [color={rgb, 255:red, 36; green, 114; blue, 18 }  ,opacity=1 ]  {$z_{0}$};
    % Text Node
    \draw (118.5,179.5) node  [color={rgb, 255:red, 1; green, 81; blue, 177 }  ,opacity=1 ]  {$y_{0}$};
    % Text Node
    \draw (281.83,130.93) node [anchor=north west][inner sep=0.75pt]    {$I( t,z,y)$};
    % Text Node
    \draw (352,265.9) node [anchor=north west][inner sep=0.75pt]    {$z$};
    % Text Node
    \draw (108,24.4) node [anchor=north west][inner sep=0.75pt]    {$y$};
    % Text Node
    \draw (191.67,225.57) node [anchor=north west][inner sep=0.75pt]    {$\left( t,z_{0} ,y'\right)$};
    % Text Node
    \draw (185,216) node [anchor=north west][inner sep=0.75pt]    {$\star$};
    % Text Node
    \draw (182.7,64) node [anchor=north west][inner sep=0.75pt]    {$+$};
    % Text Node
    \draw (190.8,36.6) node [anchor=north west][inner sep=0.75pt]    {$\left( t+\tau _{1} ,z_{0} ,y^{+}\right)$};

    \end{tikzpicture}

    \caption{Illustration of the microdynamics in the infected hosts along a single characteristic curve.}
    \label{fig:IDynamics}
    
\end{figure}
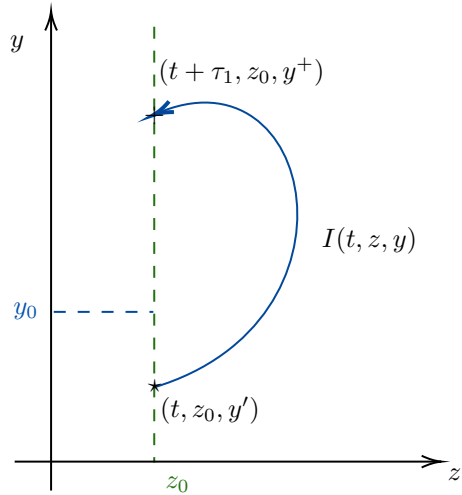

\subsection{Definition of solution to the transport equation}
The flow of the characteristics provides one way to conveniently define the solution $I$ of the two-dimensional transport equation in \eqref{eqns:FullModel}. Along the characteristic curves, by their own definition, the solution $I(t,z(t),y(t))$ satisfies
\begin{equation}\label{IODE}
    \begin{aligned}
        \frac{d}{dt} I(t,z(t),y(t)) = (a_3-a_1)I(t,z(t),y(t)).
    \end{aligned}
\end{equation}
Therefore, the value of the solution is propagated according to the above ODE. Then, given  $(t,z,y) \in [0,+\infty)\times A,$ the value of $I(t,z,y)$ is found by integration of the ODE \eqref{IODE} from an initial point of the characteristic, either at a point $(t,z_0,y')$ with $(z_0,y')$ on the inflow boundary $Z_{\mathrm{in}}$, or on a point $(0,\overline{z},\overline{y}) \in A.$ In the latter case, the initial data $I_0$ should be taken into account, while in the former case the boundary condition \eqref{eqns:BCI} should be used. We can express this using the flow of the characteristics. Using Prop.~\ref{prop:y}, we see that the flow of the characteristics defines a bijection between $(z,y)\in A$ and $(\tau,y')$ in the set
$$
\{ (\tau,y') : y'\in (0,y_0), \tau \in (0,\tau_1(y'))\}.
$$
Thus, we may write $\tau \equiv \tau(z,y),$ and 
we find
\begin{equation*}
    I(t,z,y) = \left\{\begin{aligned}
        &I\big(t-\tau(z,y),z_0,y'(z,y)\big) e^{(a_3-a_1)\tau(z,y)}, && \text{if } t-\tau(z,y) >0
        \\
        &I_0(\overline{z},\overline{y})e^{(a_3-a_1)\tau(z,y)}, && \text{if } t-\tau(z,y), \le 0,
    \end{aligned}\right.
\end{equation*}
see Section \ref{sec:FullR0} below for more details.
Using the boundary condition \eqref{eqns:BCI}, becomes
\begin{equation}\label{DefSolI}
    I(t,z,y) = \left\{\begin{aligned}
        &\frac{E(t-\tau) g(y')}{\tau_h (a_1 z_0 - a_2y^\star)} e^{(a_3-a_1)\tau}, && \text{if } t-\tau >0
        \\
        &I_0(\overline{z},\overline{y})e^{(a_3-a_1)\tau}, && \text{if } t-\tau \le 0.
    \end{aligned}\right.
\end{equation}

 % $t-\tau$ to $t$, where $\tau >0$ is such that $(t-\tau,z(t-\tau),y(t-\tau))$ is at the border of the 3-dimensional domain of definition of $I$; that is, either $t-\tau=0$, or $z(t-\tau) = z_0$, or $y(t-\tau) = 0.$ Since we are interested only in the domain of $(z,y)$ bounded by the characteristic curve originating from $(z_0,0)$ and the line $\{z=z_0\}$, we can exclude the case $y(t-\tau) = 0$ (see Fig.~\ref{fig:Icharacterists}). 
%
%
%

\subsection{Mass conservation}
\label{sec:mass}

The system \eqref{eqns:FullModel}-\eqref{eqns:ini} consists of a 2-dimensional transport equation (for $I$) coupled to a 1-dimensional transport equation (for $R$), whose source term is the trace of $I$ on its outflow boundary. The trace of $R$ at $y=y_0$ is then input in the ODE for $S$, which is coupled to $E$, $E_v$, and $I_v$. In turn, $E_v$ depends on an integral operator of $I$, and the value of $E(t)$ contributes to the boundary value of $I$ on its inflow boundary.
From this description, we see that well-posedness for the problem may pose some challenge. Our first result asserts the mass conservation of the host population. The proof is postponed to Section~\ref{sec:proofs}.
\begin{proposition}\label{prop:cons}
Let $S,E,I,R, E_v,I_v$ solve the system \eqref{eqns:FullModel} in a classical sense.
  Then, the total population of the host is conserved, that is,
  \begin{equation*}
    \frac{d}{dt} \Big(  S(t) + E(t) + \int_0^{\infty}\!\!\int_{z_0}^{\infty} {I(t,z,y)}  \,dz dy + \int_{y_0}^{+\infty} R(t,y) \,dy  \Big) = 0.
  \end{equation*}
\end{proposition}

With Proposition \ref{prop:cons} in hand, we can use the conservation of the host mass to eliminate one of the variables in the system \eqref{eqns:FullModel}. For instance, if we write the susceptible compartment $S(t)$ as

\begin{equation}
    \begin{aligned}
        S(t) = 1 - E(t) - \int_{0}^{\infty}\!\!\int_{z_0}^{\infty}I(t,z,y) \,dzdy - \int_{y_0}^{\infty} R(t,y) \,dy, \qquad t\ge 0,
    \end{aligned}
\end{equation}
we obtain the system
\begin{equation}\label{eqns:FullModel2}
  \left\{
    \begin{aligned} 
      & \frac{\partial I}{\partial t} (t,z,y) + \frac{\partial }{\partial z} \big(( a_1 z - a_2 y) I(t,z,y)  \big)
    \\
    & \qquad\qquad\qquad {}+ \frac{\partial }{\partial y} \big( (-a_3 y  + a_4 z ) I(t,z,y)  \big) = 0, \qquad z \ge z_0, y\ge 0,
      \\
      & \frac{\partial R}{\partial t} (t,y) - \frac{\partial }{\partial y} \left( a_5 y R(t,y) \right)  = - I(t,z_0, y) \bigl( a_1 z_0 - a_2 y \bigr), \qquad y \ge y_0,
      \\
      & \dot{E}(t) = b \beta_{vh} m I_v(t) \Big( 1 - E(t) - \int_{0}^{\infty}\!\!\int_{z_0}^{\infty}I (t,z,y) \,dzdy - \int_{y_0}^{\infty} R(t,y) \,dy \Big) - \frac{1}{\tau_h}E(t),
      \\
      & \dot{E}_v(t) =  b \int_0^{\infty}\!\!\int_{z_0}^{\infty} {I(t,z,y)} \beta_{hv} (z) \,dz dy   \,\big( 1 - E_v(t) - I_v(t) \big)
       - \Big(\frac{1}{\tau_v} + \mu_v\Big) E_v(t), 
      \\
      & \dot{I}_v(t) = \frac{1}{\tau_v} E_v(t)- \mu_v I_v(t),
    \end{aligned}\right.
\end{equation}
along with the boundary conditions \eqref{eqns:BCI}
and the nonnegative initial conditions
\begin{equation}
  \label{eqns:ini2}
  \left\{
    \begin{aligned}
      &I(0,z,y) = I_0(z,y), \, z\ge z_0, y\ge 0, \quad R(0,y) = R_0(y),\, y\ge y_0,
      \\
      & E(0)=E_0, \quad E_v(0) = E_{v0}, \quad I_v(0) = I_{v0},
      \\
      &E_0 + \int_{z_0}^{\infty}\!\!\int_{0}^{\infty} I_0(z,y) \,dydz + \int_{y_0}^{\infty} R_0(y) \,dy \le 1,
    \\
    &I_{v0} + E_{v0} \le 1.
    \end{aligned}\right.
\end{equation}

The system \eqref{eqns:FullModel2},\eqref{eqns:ini2} has one less equation than the original system \eqref{eqns:FullModel}, which is an advantage in the numerical simulations below, where we use \eqref{eqns:FullModel2},\eqref{eqns:ini2}. However, there is no advantage of the simpler version in what concerns the proof of the well-posedness results, and so we keep the original version \eqref{eqns:FullModel} which is slightly clearer.

\subsection{Well-posedness result}
\label{sec:well-posedn-result}

Next, we state a well-posedness result for the system \eqref{eqns:FullModel},\eqref{eqns:ini}. First, let us make precise the notion of solution we are considering. The transport equation for $I$ has a time-dependent boundary condition on its inflow boundary. We will need to consider cases where this boundary condition is merely continuous in time. The propagation along the characteristics \eqref{eqns:characteristic} does not smoothen the solution, therefore strong solutions of the $I$ equation are not guaranteed. Moreover, it is natural to consider, for instance, situations where the initial infected population is zero ($I_0(z,y)\equiv 0$), but $E_0>0.$ This instantly produces a discontinuity in $I$ for $t>0$, due to the boundary condition \eqref{eqns:BCI}. So, unless there is an artificial compatibility condition at $t=0$ between $I_0$ and the boundary data, non-smooth solutions are expected. For this reason we rely on the results of \cite{bardos1979first,colombo2015rigorous}, which assert well-posedness of distributional solutions to the two transport equations present in the system \eqref{eqns:FullModel}. 
Is is then standard to prove that this notion of solution is equivalent to the one stated in \eqref{DefSolI}.

Furthermore, the results in \cite{bardos1979first,colombo2015rigorous} give that the solutions have bounded variation and bounded traces along the boundaries and at $t=0$, attained in an $L^1$ sense. The consequence for our purposes is that the resulting notion of solution allows for manipulations in the integral formulation which produce the same results as taking the strong formulation, integrating, and doing the same manipulations. This will be true as long as we don't need to consider pointwise values of $I$ in our arguments, which is indeed the case in the technical proofs.

For $R$, the situation is simpler, even though the strong form also cannot be used  (since the source term is not necessarily continuous). Its characteristics are solutions to $\dot{y} = -a_5 y$, and so the solution can be represented either as an integration along the characteristics, or as the integral (distributional) form (see \cite{perthame}). Both these formulations admit bounded, but possibly discontinuous, source terms, and it is easy to see that the resulting solution $R$ is still continuous.

For this reason, in the proof of our well-posedness result, we slightly abuse the notation and write the transport equations in strong form to keep the presentation cleaner and more understandable. 

\begin{theorem}
  \label{thm:well-posedn-result}
  Under the assumption \eqref{eqns:29}, the system \eqref{eqns:FullModel}-\eqref{eqns:ini} has a unique solution, where the transport equations for $I$ and $R$ are understood in the distribution sense, or in the sense of \eqref{DefSolI}.
\end{theorem}

The proof of Theorem \ref{thm:well-posedn-result}, using the Banach fixed-point theorem, is conceptually standard but technically burdensome and is relegated to the supplementary materials.

\subsection{The reproduction number $\mathcal{R}_0$ and endemic equilibrium}
\label{sec:FullR0}

Here we investigate the existence of a nontrivial stationary solution to the system \eqref{eqns:FullModel2}, and relate it to a threshold parameter $\mathcal{R}_0$. This endemic equilibrium
\begin{equation}\label{endeq}
    \big(E^*,E_v^*,I_v^*,I^*(z,y),R^*(y)\big),
\end{equation}
with $E^*,E_v^*,I_v^*>0$
should satisfy the following stationary system:
\begin{equation}\label{eqns:EquiFull}
  \left\{
    \begin{aligned} 
      &\frac{\partial }{\partial z} \big(( a_1 z - a_2 y) I^*(z,y)  \big)
       + \frac{\partial }{\partial y} \big( (-a_3 y  + a_4 z ) I^*(z,y)  \big) = 0, \qquad z \ge z_0, y\ge 0,
      \\
      & - \frac{\partial }{\partial y} \left( a_5 y R^*(y) \right)  = - I^*(z_0, y) \bigl( a_1 z_0 - a_2 y \bigr), \qquad y \ge y_0,
      \\
      & 0 = b \beta_{vh} m I_v^* \Big( 1 - E^* - \int_{0}^{\infty}\!\!\int_{z_0}^{\infty}I^*(z,y) \,dzdy - \int_{y_0}^{\infty} R^*(y) \,dy \Big) - \frac{1}{\tau_h}E^*,
      \\
      & 0 =  b \int_0^{\infty}\!\!\int_{z_0}^{\infty} {I^*(z,y)} \beta_{hv} (z) \,dz dy   \,\big( 1 - E_v^* - I_v^* \big)
       - \Big(\frac{1}{\tau_v} + \mu_v\Big) E_v^*, 
      \\
      &0= \frac{1}{\tau_v} E_v^*- \mu_v I_v^*,
    \end{aligned}\right.
\end{equation}
with the inflow boundary condition for $I^*$,
\begin{equation}
  \label{eqns:BCEqui}
  \begin{aligned}
    &I^*(z=z_0, y \in [0,y_0] ) = \frac{ E^* g(y) }{\tau_h(a_1 z_0 - a_2 y^\star)} ,
  \end{aligned}
\end{equation}
where $y^\star$, the center of mass of $g$, is given by \eqref{eqns:y_star}.

The first remark is that, given any positive value $E^*$, a solution $I^*(z,y)$ of the first equation in \eqref{eqns:EquiFull} exists. Indeed, for any $(z,y)$ in the domain $A$ (cf.~\eqref{def:A}), we consider the flow of the characteristics, $\Phi(\tau,y') = (z(\tau,y'),y(\tau,y'))$, which pairs each $y'\in [0,y_0)$ and $\tau>0$ to the point $(z,y)\in A$ such that $\tau \mapsto (z(\tau,y'),y(\tau,y'))$ solves the system \eqref{eqns:ODE} with initial conditions $(z_0,y')$\footnote{Or, equivalently, $(z,y)$ are explicitly given by \eqref{eqns:characteristic}.}. Then,
\begin{equation}\label{Istar}
    I^*(z(\tau,y'),y(\tau,y')) = I(z_0,y') e^{(a_3-a_1)\tau} = \frac{E^* g(y')}{\tau_h(a_1z_0-a_2y^\star)}e^{(a_3-a_1)\tau}.
\end{equation}
The set $A$ is the set of points $(z,y)$ that can be reached by the characteristics. In particular, when $\tau = \tau_1(y')$, then according to Prop.~\ref{prop:y} the point $(z(\tau_1,y'),y(\tau_1,y'))$ lies on the outflow boundary and defines the source term for the $R^*$ equation in \eqref{eqns:EquiFull}.

Similarly, once the right-hand side of the equation for $R^*$ is well defined, then $R^*(y)$ exists under the condition $R^*(+\infty)=0$ and is given by 
\begin{equation}
    \label{Rstar}
    R^*(y) = -\frac{1}{a_5y} \int_y^{\infty}I^*(z_0, r) \bigl( a_1 z_0 - a_2 r \bigr) \,dr, \qquad y \ge y_0.
\end{equation}

With these remarks in mind, we can state our result which defines a suitable reproduction number $\mathcal{R}_0$, and relates it to the existence of an endemic equilibrium. 

\begin{proposition}\label{prop:R0Full}
    Let $\tau_1(y)$ and $y^+(y)$ be defined in Prop.~\ref{prop:y} for $y\in [0,y_0)$, and set
    \begin{equation}\label{defs1}
        \begin{aligned}     
        &\Tcal_1 = \frac{\int_{0}^{y_0}g(y)(a_1z_0 -a_2y) \tau_1(y) \,dy}{\int_{0}^{y_0}g(y)(a_1z_0 -a_2y) \,dy},
        \\
        & \Tcal_2 = \frac{\int_{0}^{y_0}g(y)(a_1z_0 -a_2y) \ln(y^+/y_0) \,dy}{a_5\int_{0}^{y_0}g(y)(a_1z_0 -a_2y) \,dy},
        \\
        &\Tcal_1[\beta_{hv}] = \frac{\int_{0}^{y_0}g(y)(a_1z_0 -a_2y) \Big[\int_{0}^{\tau_1(y)} \beta_{hv}(z(\tau,y)) d\tau\Big] \,dy}{\int_{0}^{y_0}g(y)(a_1z_0 -a_2y) \,dy}.
        \end{aligned}
    \end{equation}
    We call $\Tcal_1$ the \emph{mean infective permanence time} and $\Tcal_2$ the \emph{mean recovery time}.
    Let
    \begin{equation}\label{R_0Full}
        \mathcal{R}_0 = \frac{b^2\beta_{vh}m \,\Tcal_1[\beta_{hv}]}{\mu_v(1+\tau_v\mu_v)}.
    \end{equation}
    Then, an endemic equilibrium \eqref{endeq} of the system \eqref{eqns:FullModel2} exists if, and only if, $\mathcal{R}_0>1.$ In that case,
    \begin{equation}\label{equi-full}
      \begin{aligned}
          &E^* = \frac{\mu_v\tau_h (\mathcal{R}_0 -1)}{\mu_v \mathcal{R}_0(\Tcal_1 + \Tcal_2 + \tau_h)+ b\beta_{vh}\Tcal_1[\beta_{hv}]}, 
          \\
          &E_v^* = \frac{\mu_v\tau_v(\mathcal{R}_0-1)}{\mathcal{R}_0(1+\mu_v\tau_v) +  b \beta_{vh}m(\Tcal_1 + \Tcal_2 + \tau_h)},
          \\
          &I_v^* = \frac{(\mathcal{R}_0-1)}{\mathcal{R}_0(1+\mu_v\tau_v) +  b \beta_{vh}m (\Tcal_1 + \Tcal_2 + \tau_h)},
      \end{aligned}
  \end{equation}
  and $I^*,R^*$ are defined in the remarks before the statement of the proposition (see \eqref{Istar}).
\end{proposition}

Note that the mean infective permanence time $\Tcal_1$ is a weighted average of the infective permanence times $\tau_1(y)$ for $y\in (0,y_0)$. Also, The term $\Tcal_1[\beta_{hv}]$ reduces to $\Tcal_1 \beta_{hv}$ in the case where $\beta_{hv}(z)$ is a constant.

The proof of Prop.~\ref{prop:R0Full} is postponed until Section~\ref{sec:proofs}.

% \subsection{The reproduction number $R_0$ and the stability of the disease-free equilibrium}

% In this section, we show how the reproduction number \eqref{R_0Full} is related to linear stability properties of the disease-free equilibrium (DFE),
% \begin{equation}
%     (E,E_v,I_v,I,R) = (0,0,0,0,0).
% \end{equation}
% Although we could not achieve a fully rigorous result, we can still see that under some simplifications, $R_0 <1$ is indeed closely related to the linear stability of the DFE. 

% Our reasoning is as follows: consider the linearization of the system \eqref{eqns:FullModel2} around the DFE. The equations for $I$ and $R$ do not change, as they are linear already. The equations for $E,E_v,I_v$ are
% \begin{equation}\left\{
%     \begin{aligned}
%         &\dot E(t) = b\beta_{vh}m I_v - \frac{1}{\tau_h}E,
%         \\
%         & \dot E_v(t) = b \int_0^{\infty}\!\!\int_{z_0}^{\infty} {I(t,z,y)} \beta_{hv} (z) \,dz dy - \Big(\frac{1}{\tau_v}+ \mu_v\Big) E_v,
%         \\
%         & \dot I_v(t) = \frac{1}{\tau_v} E_v - \mu_v I_v.
%     \end{aligned}\right. 
% \end{equation}

%
%
%
%%%%%%%%%%%%%%%%%%%%%%%%%%%%%%%%%%%%%%%%%%%%%%%%%%%%%%%%%%%%%%%%%%
\section{A model with uniform host response}
\label{subsec:DelayModel}

The full model \eqref{eqns:FullModel} poses numerical challenges due to the two-dimensional nature of the transport equation. In this section, we deduce from \eqref{eqns:FullModel} a particular case where we assume that the probability distribution $g(y)$ in \eqref{eqns:BCI}, which governs the antibody distribution of the population as they enter the infected compartment from the exposed compartment, is a Dirac delta concentrated on some $y^\star \in (0,y_0)$. This means that there is no uncertainty in the antibody level of individuals as they enter the infected compartment, which implies that the within-host dynamics are the same for all the host population. We call this the \emph{uniform host response} model, or UHR model for short. A uniform response might be a good approximation when the host population is small, or when the virus-antibody dynamics does not have large variations in the population. 
Thus, we assume that for some $y^\star\in (0,y_0)$, $g$ is the distribution on $(0,y_0)$
\begin{equation}
  \label{eqns:g}
  g(y) = \delta(y-y^\star).
\end{equation}

Let us now see what consequences this choice has on the structure of the system \eqref{eqns:FullModel}.
At least formally, the choice \eqref{eqns:g} implies that the infected host population is concentrated on the single characteristic issuing from $(z_0,y^\star)$, namely \eqref{eqns:characteristic}. That is, the microdynamics of the infected host is described by a single characteristic curve, starting from $(z_0,y^\star)$. 
(Actually, this can be made precise, as in \cite[Prop.6.4]{perthame}, and we can deduce that $I(t,z,y) = \delta\big((z,y) - (z(t),y(t))\big)$ is a solution in the distributional sense to the transport equation for $I$; but this falls outside the scope of this work).

Under the present assumptions, the total outflow from the infected compartment $I$ at time $t+\tau_1$ (where $\tau_1 = \tau_1(y^\star)$ is given by Prop.~\ref{prop:y}) on the border $Z_{\mathrm{out}}$ is the same as the inflow at time $t$ at the inflow border $Z_{\mathrm{in}}$, which is controlled by the boundary condition \eqref{eqns:BCI}. At the same time, at the outflow border, $I(t,z_0,y)$ should be a multiple of the Dirac delta concentrated at $y^+=y^+(y^\star)$ (see Prop.~\ref{prop:y}). These facts imply that 
\begin{equation}
   (a_1 z_0- a_2 y) I(t+\tau_1,z_0, y ) = -\frac{E(t)}{\tau_h} \delta(y-y^+)
  \label{eq:I_tau1}
\end{equation}
on the border $Z_{\mathrm{out}}$, which is the domain of definition of $R(t,y)$.
Therefore, we (formally) have for $R(t,y)$
\begin{equation}
  \label{eqns:30}
  \begin{aligned}
    \frac{\partial R}{\partial t}(t,y) - \frac{\partial }{\partial y}(a_5 y R(t,y)) =-\frac{E(t)}{\tau_h} \delta(y-y^+), \qquad t>\tau_1,\, y\in [y_0,+\infty).
  \end{aligned}
\end{equation}
Now, it is a simple exercise in distribution theory to show that, in the present context, the equation \eqref{eqns:30} with a Dirac delta on the right-hand side, concentrated on a (fixed) $y^+ > y_0$, is equivalent to solving the boundary value problem
\begin{equation}
  \label{eqns:31}
  \left\{
    \begin{aligned}
      & \frac{\partial R}{\partial t}(t,y) - \frac{\partial }{\partial y}(a_5 y R(t,y)) = 0, \qquad t>\tau_1, \quad y\in [y_0,y^+],
      \\
      & R(t,y^+) = \frac{1}{a_5 y^+ \tau_h}E(t-\tau_1), \qquad t>\tau_1.
    \end{aligned}\right.
\end{equation}
From now on, we suppose $R$ is only defined for $y\in [y_0,y^+]$ to be consistent with the remaining exposition.
The equation \eqref{eqns:31} makes it easy to determine the value of $R(t,y_0)$ as a function of the inflowing value $R(t,y^+)$ at an earlier time. Indeed, an easy calculation using the characteristics of the equation \eqref{eqns:31},
\begin{equation*}
  y(t) = y(0)e^{-a_5 t}, \qquad y(0) \in [y_0,y^+],
\end{equation*}
shows that for $y\in[y_0,y^+]$ it holds
\begin{equation*}
  R(t,y(t)) = R(t-s,y(t-s)) e^{a_5s},
\end{equation*}
for $s$ such that $y(t-s) \in [y_0,y^+]$. When $y(t-s) = y^+$, then $y(t) = y^+ e^{-a_5 s}$ and so for $y\in [y_0,y^+]$ we have
\begin{equation}
  \label{eqns:10}
  \begin{aligned}
    R(t,y) &= R\Big(t- \frac{1}{a_5}\ln(y^+/y),y^+\Big) \frac{y^+}{y}
    \\
           & = \frac{1}{a_5 y^+ \tau_h} E\Big(t- \tau_1 - \frac{1}{a_5}\ln(y^+/y)\Big) \frac{y^+}{y}.
  \end{aligned}
\end{equation}

From the formula \eqref{eqns:10} we get
\begin{equation*}
  \int_{y_0}^{y^+} R(t,y) \,dy = \int_{y_0}^{y^+} \frac{1}{a_5\tau_h} E\Big(t- \tau_1 - \frac{1}{a_5}\ln(y^+/y) \Big) \frac{1}{y} \,dy. 
\end{equation*}
by a change of variables, and defining the \emph{recovery period}
\begin{equation}
  \label{eqns:32}
  \tau_2 := \frac{1}{a_5}\ln \frac{y^+}{y_0},
\end{equation} 
we arrive at
\begin{equation}\label{R01}
  \int_{y_0}^{y^+} R(t,y) \,dy = \int_{0}^{\tau_2} \frac{1}{\tau_h} E\big(t- \tau_1 - \tau_2 + s \big) \,ds = \int_{t-\tau_1-\tau_2}^{t-\tau_1} \frac{1}{\tau_h} E(s) \,ds, 
\end{equation}
for the total population in recovery at time $t$.

Going further, we assume that the host-vector infectivity probability $\beta_{hv}$ is constant. 
The total infected population at time $t$ is given by the population that entered at a rate $E(t)/\tau_h$ through the inflow (concentrated at the point $(z_0,y^\star)$), according to the boundary condition \eqref{eqns:BCI}, from time $t-\tau_1$ to $t$. Therefore we can write
\begin{equation}\label{I01}
  \begin{aligned}
    \int_{z_0}^{\infty}\!\!\int_{0}^{\infty}I(t,z,y) \,dydz  =  \int_{t-\tau_1}^{t} \frac{1}{\tau_h} E(s) \,ds. 
    \end{aligned}
\end{equation}

Using \eqref{I01} and \eqref{R01}, we can eliminate the equation for $S$ by
\begin{equation*}
  \begin{aligned}
    S(t) &= 1-E(t)- \int_{z_0}^{\infty}\!\!\int_{0}^{\infty}I(t,z,y) \,dydz - \int_{y_0}^{y^+} R(t,y) \,dy
    \\
    & = 1 -E(t) - \int_{t-\tau_1-\tau_2}^{t} \frac{1}{\tau_h} E(s) \,ds. 
    \end{aligned}
\end{equation*}

All in all, the previous arguments allow us to reduce the system \eqref{eqns:FullModel} to the following integro-differential distributed delay-type system:
\begin{equation}
  \label{eqns:delay}
  \left\{
  \begin{aligned}
    &\dot E(t) = b\beta_{vh} m I_v(t) \Big( 1 -E(t) - \int_{t-\tau_1-\tau_2}^{t} \frac{1}{\tau_h} E(s) \,ds\Big) - \frac{1}{\tau_h}E(t)
    \\
    &\dot E_v(t) = b\beta_{hv}\int_{t-\tau_1}^{t} \frac{1}{\tau_h} E(s) \,ds \cdot \big(1 - E_v(t) - I_v(t)\big)
    - \Big(\frac{1}{\tau_v} + \mu_v \Big) E_v(t)
    \\
    &\dot I_v(t) = \frac{1}{\tau_v}E_v(t) - \mu_v I_v(t).
  \end{aligned}\right.
\end{equation}

Due to the distributed delay nature of the system, the initial data for $E$ should include information about the past behavior of $E$. Therefore,
the data for the system \eqref{eqns:delay} are
\begin{equation}
  \label{eqns:data}\left\{
  \begin{aligned}
    &E_v(0) = E_{v0},\quad I_v(0) = I_{v0},
    \\
    & E(t) = E_0(t), \quad t \in [-\tau_1 - \tau_2,0],
  \end{aligned}\right.
\end{equation}
subject to 
\begin{equation}
    E_{v0} + I_{v0} \le 1, \qquad \int_{-\tau_1-\tau_2}^{0} E_0(s) \,ds \le 1.
\end{equation}
Note that in this formulation, the information about the initial recovered population is contained in the values of $E_0(s)$ for $s\in [-\tau_1-\tau_2, -\tau_1],$ and the information about the initial infected host population is expressed in $E_0(s)$ for $s\in [-\tau_1,0].$ In particular, the integral of $E_0(s)$ over these two intervals gives, respectively, the total initial recovered host and the total initial infected host.

%%%%%%%%%%%%%%%%%%%%%%%%%%%%%%%%%%%%%%%%%%%%%%%%%%%%%%%%%%%%%%%%%%
\subsection{The reproduction number and endemic equilibrium for the uniform host response model}
\label{sec:Dynamics}
%%%%%%%%%%%%%%%%%%%%%%%%%%%%%%%%%%%%%%%%%%%%%%%%%%%%%%%%%%%%%%%%%%

The disease-free equilibrium (DFE) of the system \eqref{eqns:delay} is given by
\begin{equation}
  \mathrm{DFE} = (E, E_v, I_v) = \left(0, 0, 0 \right),
\end{equation}
where $E =0$ means that $E(t)=0$ for all $t\in[-\tau_1-\tau_2,0]$.

The endemic equilibrium $(E^*, E_v^*, I_v^*)$
of the system \eqref{eqns:delay} is described in the following proposition, where we also introduce the endemic threshold, or basic reproduction number, $\mathcal{R}_0$. The result is to be compared with the results in Prop.~\ref{prop:R0Full} for the full system.
\begin{proposition}\label{prop:R0}
Let
\begin{equation}\label{R0}
     \mathcal{R}_0 =  \frac{b^2 \beta_{vh} \beta_{hv} m \tau_1}{ \mu_v ( \mu_v \tau_v +1 ) }.
\end{equation}
  Then, there exists an endemic equilibrium $(E^*,E_v^*,I_v^*) \in (0,1)^3$ of the system \eqref{eqns:delay} if and only if $\mathcal{R}_0>1.$ In that case, we have
  \begin{equation}\label{equi-delay}
      \begin{aligned}
          &E^* \equiv \frac{\mu_v\tau_h (\mathcal{R}_0 -1)}{\mu_v \mathcal{R}_0(\tau_1+\tau_2+\tau_h)+ b\beta_{hv}\tau_1}, \qquad \text{on }  [-\tau_1-\tau_2,0],
          \\
          &E_v^* = \frac{\mu_v\tau_v(\mathcal{R}_0-1)}{\mathcal{R}_0(1+\mu_v\tau_v) +  b \beta_{vh}m(\tau_1+\tau_2+\tau_h)},
          \\
          &I_v^* = \frac{(\mathcal{R}_0-1)}{\mathcal{R}_0(1+\mu_v\tau_v) +  b \beta_{vh}m(\tau_1+\tau_2+\tau_h)}.
      \end{aligned}
  \end{equation}
\end{proposition}
\begin{proof}
    The fact that \eqref{equi-delay} are stationary states of the system \eqref{eqns:delay} is a direct calculation, in which we only have to observe that the integral terms give, for instance,
    \begin{equation}
        \int_{t-\tau_1-\tau_2}^t \frac{1}{\tau_h}E^* \,ds = E^*\frac{\tau_1+\tau_2}{\tau_h}.
    \end{equation}
    The positivity of the endemic equilibrium follows immediately from $\mathcal{R}_0>1$, while $E^*,E_v^*,I_v^* <1$ follow easily from the positivity of $\mathcal{R}_0$. We omit the (straightforward) details.
\end{proof}

\begin{remark}
    As we have seen, the uniform response model \eqref{eqns:delay} is the formal limit of the full system \eqref{eqns:FullModel} when the transition function $g(y)$ from the exposed to the infective compartment becomes a Dirac delta. We remark here that this limit can also be observed in the expression of $\mathcal{R}_0$ and of the equilibrium \eqref{equi-delay} and \eqref{R0}, in relation to the corresponding quantities for the full model, namely \eqref{equi-full} and \eqref{R_0Full},\eqref{defs1}. Indeed, when $g(y)$ is a Dirac delta (necessarily centered at $y^\star$), then the expression of the mean infectious permanence time $\Tcal_1$ in \eqref{defs1} reduces to $\tau_1(y^\star)$, which is the infective permanence time $\tau_1$ appearing in the UHR model. The same holds for the mean recovery time $\Tcal_2$ and the recovery period $\tau_2$ in \eqref{eqns:32}. In this way, we see that the expressions in \eqref{R0} and \eqref{equi-delay} are obtained as the limits of the corresponding quantities for the full system in the limit $g\to \delta_{y^\star}.$
\end{remark}

 \begin{remark}
    Let us recall that the foundational interpretation of the basic reproduction number is that ${R}_0$ corresponds to the {\em number of secondary infections caused by a single infected individual of the same type (host or vector) during its infectious period in a completely susceptible population} \cite{macdonald1952analysis}. With this in mind, we obtain an intuitive derivation of ${R}_0$ for the system \eqref{eqns:delay}, following the ideas presented in \cite{lopez2002threshold}. We find a similar threshold to the one established in our study. In fact, the expected number of infected mosquitoes generated by a human host as index case, is given by the product of the mosquito-to-human ratio $m$, the mosquito biting rate $b$, and the probability of transmission from the host to the vector $\beta_{hv}$ during its infectious period $\tau_1$. This quantity is expressed by $ m b \beta_{hv} \tau_1 $.
    These infected mosquitoes, during their average lifespan $1/\mu_v$, generate new host infections. These new infections depend on the mosquito biting rate $b$, the probability of transmission from vector to host $\beta_{vh}$, and the fraction of mosquitoes that survive the extrinsic incubation period $e^{-\mu_v \tau_v}$. Accordingly, the expected number of secondary human infections is given by
    \begin{equation*}
        m b \beta_{hv} \tau_1 \frac{b \beta_{vh} e^{-\mu_v \tau_v}}{\mu_v} = \frac{b^2 \beta_{vh} \beta_{hv} m \tau_1}{ \mu_v e^{\mu_v \tau_v} }.
    \end{equation*}
    Note that this expression differs from the basic reproduction number $\mathcal{R}_0$, defined in \eqref{R0}, only through a first-order Taylor approximation of the exponential term $e^{\mu_v \tau_v}$.
\end{remark}

\subsubsection{$\mathcal{R}_0$ and the stability of the disease-free equilibrium}
\label{sec:DFEstab}

So far, the reproduction number $\mathcal{R}_0$ acts as a threshold at $\mathcal{R}_0=1$ in the sense that the endemic equilibrium exists if, and only if, $\mathcal{R}_0>1$. This is true both for the full model (Prop.~\ref{prop:R0Full}) and for the UHR model (Prop.~\ref{prop:R0}). In fact, another desirable property of $\mathcal{R}_0$ would be that the disease-free equilibrium is asymptotically stable if, and only if, $\mathcal{R}_0<1.$ Although we were unable to prove such a result, we can make some comments. The UHR system can be seen as a delay-type system, and as such the study of the characteristic polynomial of the linearized system around the DFE yields a transcendental equation. Omitting the details for the sake of brevity, we are led to analyze the roots $\lambda$ of the equation
\begin{equation}\label{char-equation}
\begin{aligned}
     0 & =   \tau_h \tau_v \lambda^3 + \left( 2 \tau_h \tau_v \mu_v + \tau_h + \tau_v \right) \lambda^2 + \left(  \tau_h \tau_v \mu_v^2 +  \tau_h\mu_v + 2\tau_v\mu_v + 1 \right) \lambda 
     \\
     &  \quad +\left( \tau_v \mu_v^2 + \mu_v  \right)  + b^2 \beta_{vh} \beta_{hv} m   \frac{e^{-\lambda \tau_1} - 1}{\lambda}.   
     \end{aligned}
\end{equation}
From \eqref{char-equation} it is easy to see that, any \emph{real} root is negative if, and only if, $\mathcal{R}_0<1,$ which at least suggests that the desired stability threshold property holds. But, of course, one cannot exclude complex roots $\lambda$, and in that situation one obtains a transcendental equation for which we could not prove that $\Re(\lambda)<0$ whenever $\mathcal{R}_0<1.$

Still, all our numerical experiments, reported below, indicate that the threshold property of $\mathcal{R}_0$ holds, not only for the UHR model, but also, in fact, for the full model.

\section{Numerical Simulations}
\label{sec:Numerical}
%%%%%%%%%%%%%%%%%%%%%%%%%%%%%%%%%%%%%%%%%%%%%%%%%%%%%%%%%%%%%%%%%%

\subsection{Influence of within-host dynamics on epidemiological outcomes}

Consider first the uniform host response model \eqref{eqns:delay}. Recall that the reproduction number (which we repeat here for convenience) was found in Prop.~\ref{prop:R0} to be
\begin{equation}\label{R0a}
     \mathcal{R}_0 =  \frac{b^2 \beta_{vh} \beta_{hv} m \tau_1}{ \mu_v ( \mu_v \tau_v +1 ) }.
\end{equation}

Recall that the mean host infective permanence period $\tau_1$ was obtained in Proposition~\ref{prop:y} as the time measured along the (single) characteristic curve along which the infected compartment ``travels'' in the uniform host response model. $\tau_1$ is a function of the parameters appearing in~\eqref{eqns:tau1Characteristics}, namely $a_{1,2,3,4}$. Still, the $a_i$ do not appear in the uniform host response system \eqref{eqns:delay}. Therefore, $\tau_1$ can also be seen as a free parameter of that model.
In that case, the reproduction number $\mathcal{R}_0$ will be directly proportional to $\tau_1$, when the other parameters are kept fixed. This is in line with the expectation that a longer infectious period (i.e., a larger $\tau_1$) will worsen the epidemic outcome by increasing $\mathcal{R}_0$.

On the other hand, investigating the influence of the parameters $a_i$ on $\tau_1$ (and, by consequence, on $\mathcal{R}_0$) would help to understand how within-host dynamics can influence epidemiological outcomes in the context of our model. This can be useful in two complementary ways. First, as observed throughout this work, the within-host dynamics proposed are phenomenological and not intended to accurately reproduce physiological processes, but only the broad dynamics involved. In this sense, it is a challenging task to determine reasonable baseline values for $a_i$. But $\tau_1$ is an epidemiological parameter which can be experimentally determined for many diseases; therefore, taking $\tau_1$ (or $\mathcal{R}_0$) as a known value and adjusting the $a_i$ accordingly can be a good way to calibrate the parameters associated with the within-host dynamics. Conversely, the $a_i$ represent, in broad strokes, how the virus interacts with the immune system in the host throughout the pathogenesis process. Knowing which of their values has more or less influence on epidemiological variables such as $\mathcal{R}_0$ can help to determine areas of focus for research.

To further explore these questions, we investigate numerically the influence of $a_{1,2,3,4}$ on the value of the reproduction number $\mathcal{R}_0$.
We recall here the following facts from Table~\ref{tab:ModelParameters}:
\begin{itemize}
    \item $a_1$ is the virus growth rate in the host,
    \item $a_2$ is the viral decay rate per unit of antibody,
    \item $a_3$ is the background antibody decay rate, and
    \item $a_4$ is the antibody production rate per unit of viral load. 
\end{itemize}
According to these interpretations, it should be a feature of the model that:
\begin{itemize}
    \item $\mathcal{R}_0$ should increase with $a_1$ and $a_3$, and
    \item $\mathcal{R}_0$ should decrease when $a_2$ and $a_4$ increase.
\end{itemize}

In the following experiments we show that the model reproduces these expectations.

\subsubsection*{Baseline parameters}

We collect in Table~\ref{tab:parameters} realistic baseline values for the parameters that do not involve  within-host dynamics. Most of these have well-established ranges in the literature, see for instance \cite{massad2010estimation,feng1997competitive,kribs2025impact,gulbudak2020infection}.

\begin{table}[ht]
    \centering
    \caption{Parameters and corresponding baseline value for the system \eqref{eqns:delay}.}
    \label{tab:parameters}
    \begin{tabular}{lll}
    \hline
    Parameter & Description & Baseline value \\
    \hline
        $b$ & biting rate & 1/3\\
        $\beta_{vh}$ & Vector to host transmission probability & 0{,}25\\
        $\beta_{hv}$ & Host to vector transmission probability & 0{,}2\\
        $m$ & Mosquito-to-host ratio & 6\\
        $\mu_{v}$ & Natural vector death rate & 0{,}05\\
        $\tau_h$ & intrinsic incubation period & 7\\
        $\tau_v$ & extrinsic incubation period & 10\\
        $\tau_1$ & Infectious period & $\sim$ 7\\
        $\mathcal{R}_0$ & Basic reproduction number & $\sim$ 3\\
        \hline
    \end{tabular}
\end{table}
% \DM{
% $b$, $m$, $\tau_1 = 1/\gamma$, $\tau_v$, $\mu_v$ are similar from \cite{massad2010estimation}, but its not a numerical paper, just data based on real cases.
% %
% \cite{feng1997competitive} presents some parameters for dengue fever, but they differs from ours.
% % A Systematic Review of Mathematical Models of Dengue Transmission and Vector Control: 2010–2020 --> Nao tem data nem parametros mas lista referencias no assunto
% %
% Check references from table 1.
% %
% \cite{kribs2025impact} present similar values for $\mu_v$.
% %
% {zahid2023} estimates parameters from data.
% %
% \cite{gulbudak2020infection} similar $\mu_v$.
% }

As explained above, we look for values of $a_1,\dots,a_4$ such that the resulting $\tau_1$ is close to the reference value in Table~\ref{tab:parameters}. Since $\mathcal{R}_0$ is also a given, this imposes a constraint on the combination of the remaining parameters that appear in the expression of $\mathcal{R}_0$, namely,
\begin{equation}
    \frac{b^2 \beta_{hv}\beta_{vh}m}{\mu_v(1+\tau_v\mu_v)} = \frac{\mathcal{R}_0}{\tau_1} \simeq \frac{3}{7} = 0.429\dots.
\end{equation}

This constraint is satisfied using the values in Table~\ref{tab:parameters}. The set of admissible  $(a_1,a_2,a_3,a_4)$ is then the set $D_a\subset \RR^4$ such that $4a_2a_4 > (a_1+a_3)^2$ (cf.~\eqref{eqns:beta}) and $\tau_1 \in (5,9).$ Since $\tau_1$ is a continuous function of $(a_1,a_2,a_3,a_4)$ (from classical results on the continuous dependence of ODEs with respect to parameters), $D_a$ is an open set. However, it is difficult to gain insight into its structure. For that reason, we performed a random search
 among the 4-vectors $(a_1,a_2,a_3,a_4)$ in the domain $(0, 10)^4$ to select values of the $a_i$ in $D_a$. In this way we produced a set $D \subset D_a$ with {100} points.
 
 We performed a K-Means clustering analysis of the set of 4-vectors $D$, but did not find significant multimodal clustering. This suggested taking the centroid of $D$ as a representative value of the set of admissible $a_i$. However, since no convexity property of $D_a$ is guaranteed, this generally did not yield a point belonging to $D_a$.

Therefore, we chose the baseline point $(a_1,a_2,a_3,a_4)$ by selecting from the set of samples $D$ one yielding a numerical solution spiraling towards the equilibria of the system, and such that the characteristic curves were easily visualized:
\begin{equation}\label{eqns:ref_as}
    (a_1,a_2,a_3,a_4)_\text{ref} = (1,\ 0.44,\,0.72,\,1.93).
\end{equation}

The resulting infective period $\tau_1$ is $6.4$ days, the reproductive number is $\mathcal{R}_0 = 2.86,$ and the recovery period is $\tau_2 =70$ days with $a_5=0.01$. 

\subsubsection*{Reference solutions}
In Fig.~\ref{fig:refsol} we plot the time series of the solution to the UHR system with the parameters of Table~\ref{tab:parameters}. We can observe the convergence of the solution to the endemic equilibrium values (horizontal lines) predicted by Prop.~\ref{prop:R0}.

In Figs.~\ref{fig:refsolfull} and \ref{fig:ref_full_infected} we plot the solution to the full system \eqref{eqns:FullModel2}. To simulate the transport equation for $I$, we implemented a first-order upwind method, on a $200 \times 340$ grid on the $(z,y)$ computational domain shown in Fig.~\ref{fig:ref_full_infected}. The resulting numerical solution is known to be too dissipative \cite{leveque2002finite}, and so it is difficult to observe clearly the transport of the infective compartment along the direction of the characteristic curves, also plotted in Fig.~\ref{fig:ref_full_infected}. We choose $g(y)$ as a gaussian with variance $0.2$, centered on the midpoint of the interval $(0,y_0=2.27)$, and we can see that the time evolution of the compartments (Fig.~\ref{fig:refsolfull}) is qualitatively similar to the solution of the UHR system (Fig.~\ref{fig:refsol}). The smearing out of the $I$ compartment, evident in Fig.~\ref{fig:ref_full_infected}, can be minimized by using high-resolution schemes for hyperbolic equations \cite{leveque2002finite}, although we leave this to future work. We also plot the equilibrium values predicted by Prop.~\ref{prop:R0Full}. The slight discrepancy observed (especially in the $R$ component) may be ascribed to the diffusive nature of the numerical approximation of the $I$ and $R$ equations, and to discretization errors in the integral terms present in the expressions of the equilibria \eqref{equi-full}. In particular, it is natural that $I$ (and, by consequence, $R$) is underestimated by our numerical procedure, since the numerical diffusion may lead to mass leakage along the border of the computational domain. 

None of the previous drawbacks apply to the simulation of the UHR system, and so the remaining numerical experiments will focus mainly on its simulation. At the same time, the greater versatility of the full system, in the sense that it can account for uncertainty in the infective and recovery times of the population, suggests that it is worthwhile to develop more accurate numerical methods for its simulation.

\begin{figure}
    \centering
    $$\includegraphics[width=1\linewidth]{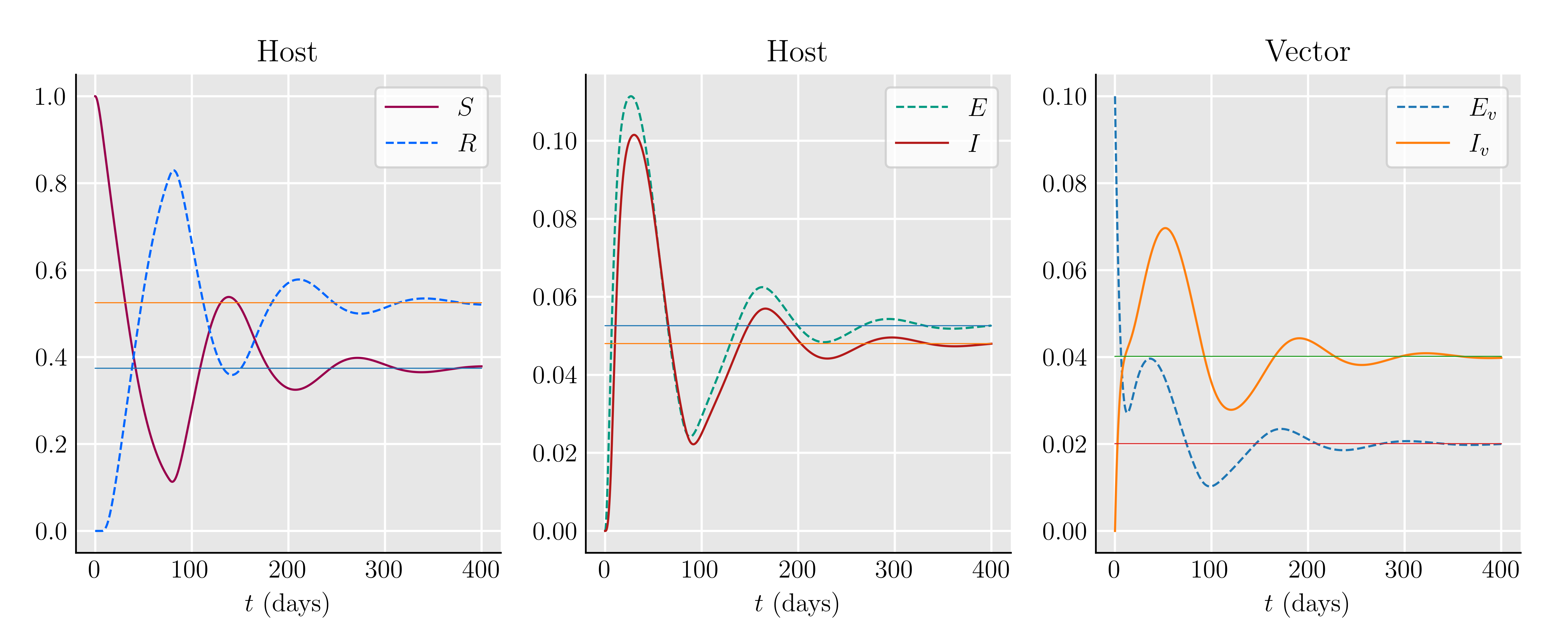}$$
    \caption{Reference solution time series for the Uniform Host Response model \eqref{eqns:delay}. The parameters used are in Table~\ref{tab:parameters}, along with \eqref{eqns:ref_as} and $a_5=0.01.$ The host infective period $\tau_1$ is $6.4$ days, the recovery period is $\tau_2=70$ days, and the reproductive number $\mathcal{R}_0$ is $2.86$. Horizontal lines are the equilibrium values.}
    \label{fig:refsol}
\end{figure}
\begin{figure}
    \centering
    \includegraphics[width=1\linewidth]{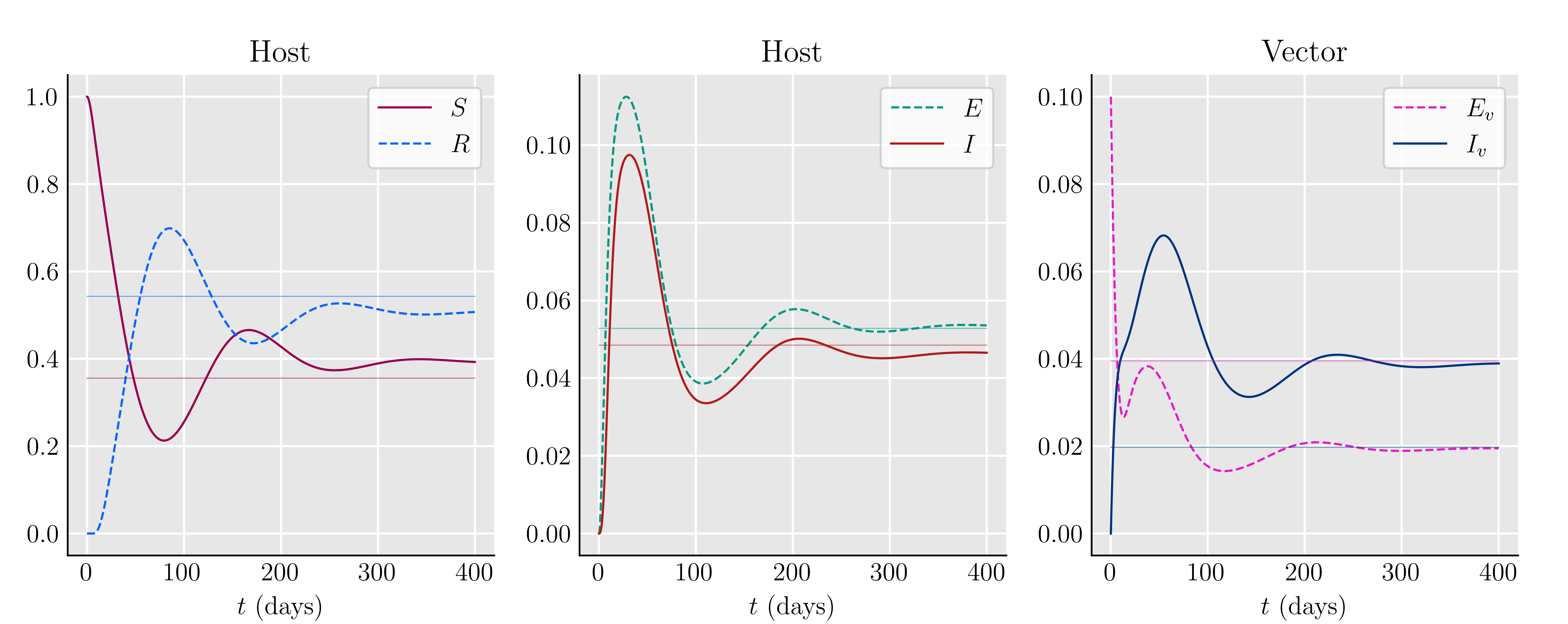} 
    \caption{Reference solution for the full model \eqref{eqns:FullModel2}. The parameters used are in Table~\ref{tab:parameters}, along with \eqref{eqns:ref_as} and $a_5=0.01.$ Here, $R$ and $I$ represent the integrals of $R(t,y)$ and $I(t,z,y)$ for each $t>0$. }
    \label{fig:refsolfull}
\end{figure}
\begin{figure}
    \centering
    \includegraphics[width=0.7\linewidth]{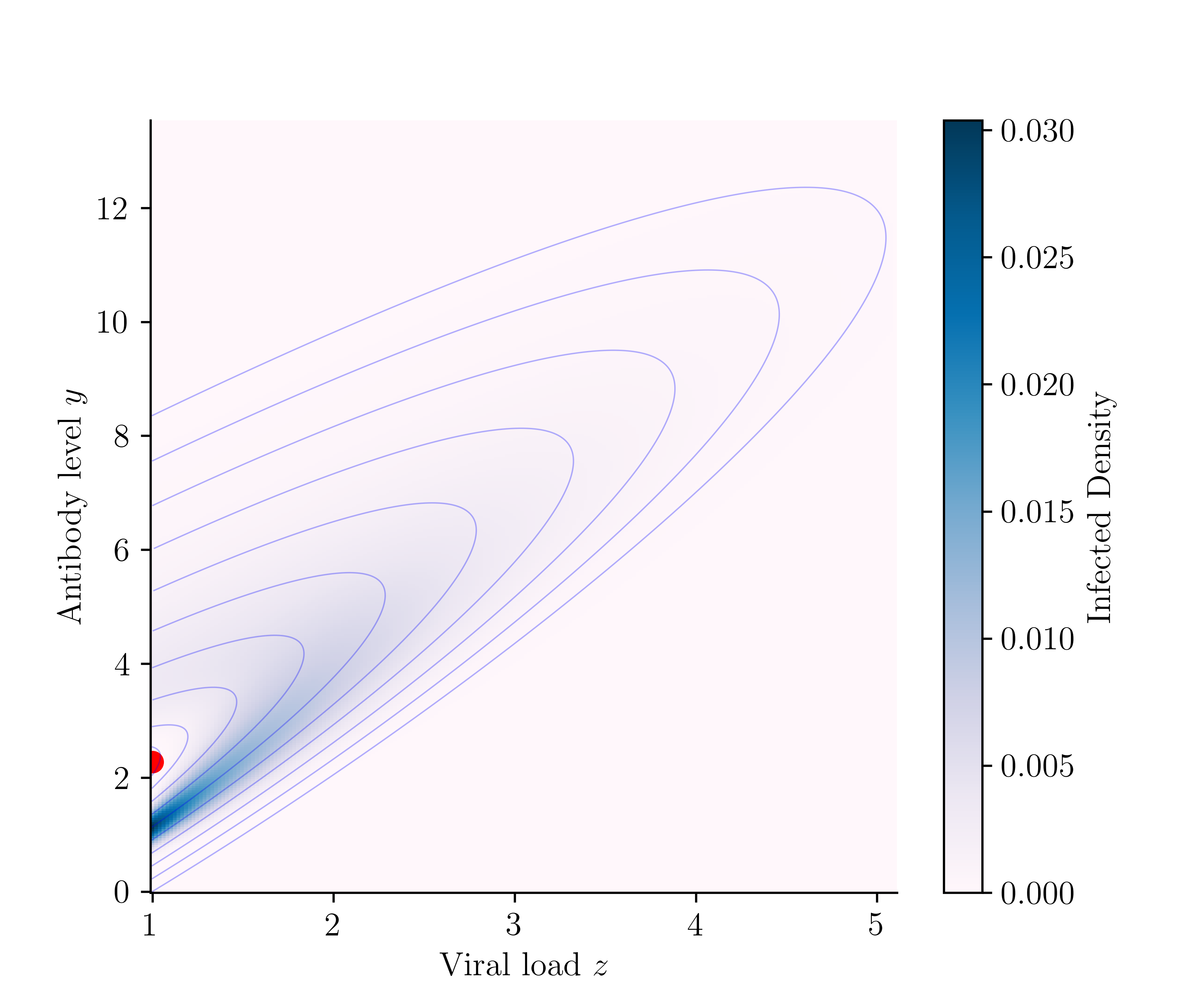}
    \caption{Density of the infective population $I(t,z,y)$ at the final time $t=400$ days for the reference solution of the full model \eqref{eqns:FullModel2}. The blue lines are characteristic curves of the transport equation for $I$.}
    \label{fig:ref_full_infected}
\end{figure}

\subsubsection*{Influence of $a_1,a_3$}

\begin{figure}
\centering
$$\includegraphics[width=0.5\linewidth]{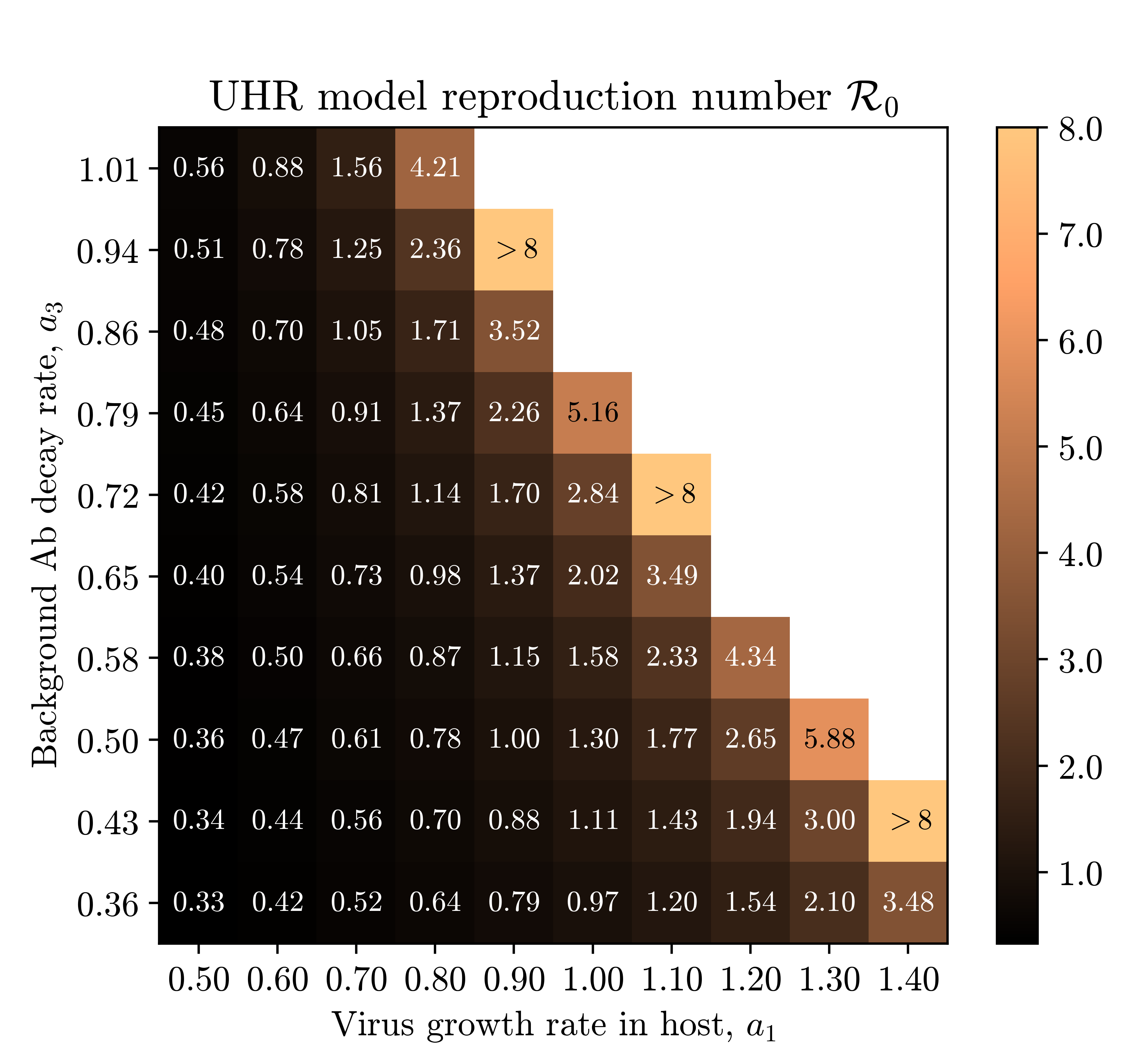}
\includegraphics[width=0.5\linewidth]{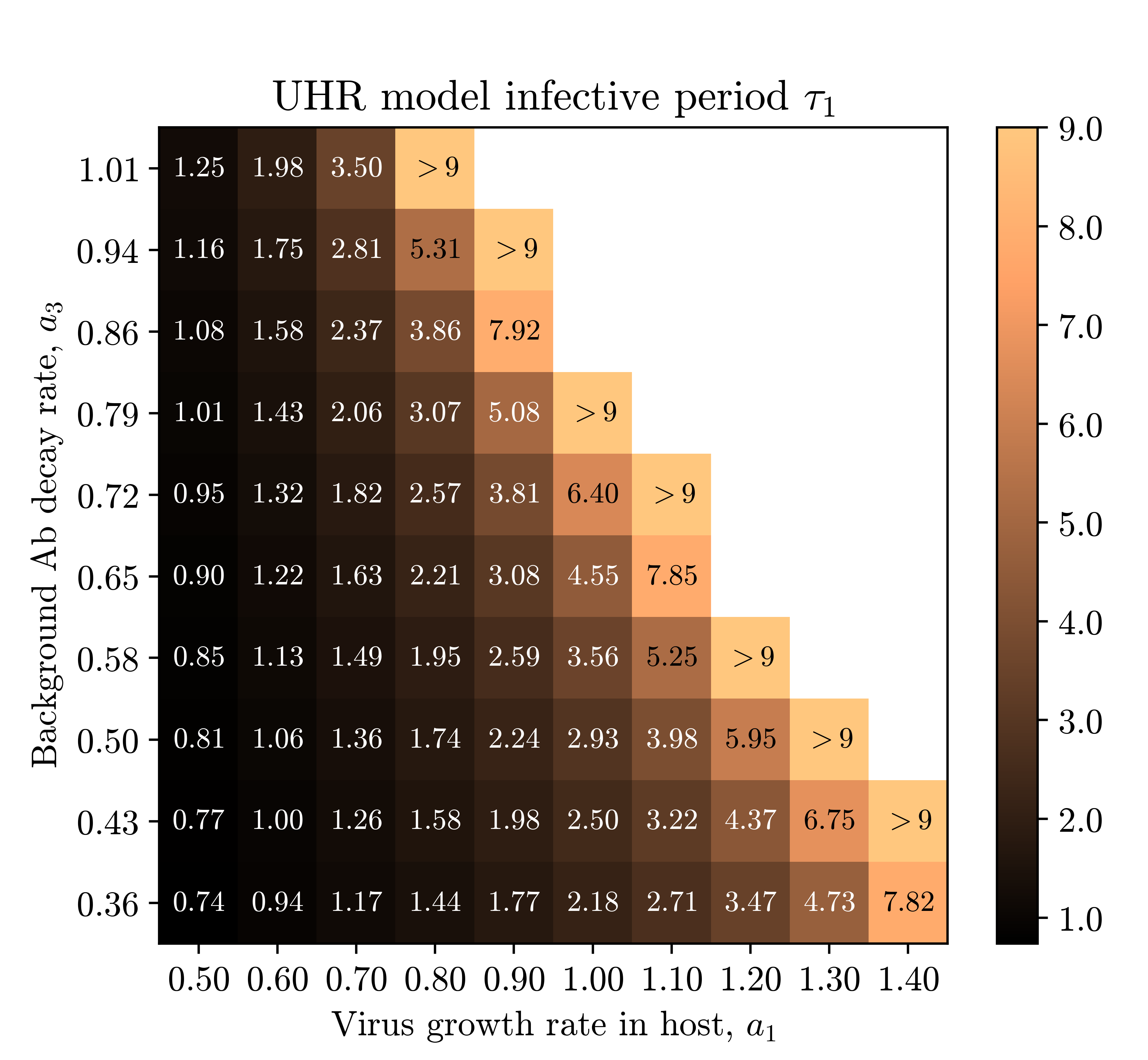}$$
\caption{$\mathcal{R}_0$ and $\tau_1$ for the UHR model as a function of $a_1$ and $a_3$, perturbed around the reference solution. The blank region corresponds to non-admissible $a_1,a_3$ values such that $\beta \not\in \RR$ (cf.~\eqref{eqns:beta}).}
\label{fig:R0a1a3}
\end{figure}

In Fig.~\ref{fig:R0a1a3}, we plot the values of $\mathcal{R}_0$ and the infective period $\tau_1$ for the UHR model as a function of $a_1$ and $a_3$. The parameters used are the baseline ones in Table~\ref{tab:parameters}, and the values of $a_1$ and $a_3$ are perturbations around the reference value in \eqref{eqns:ref_as}. We can observe that the value of $\mathcal{R}_0$ increases when $a_1$ and $a_3$ increase, but interestingly, the dependence of $\mathcal{R}_0$ on the Ab decay rate $a_3$ is weaker, especially for lower values of $a_1$, in the parameter ranges considered. This suggests that for small values of $a_1$ (that is, very low viral growth rate), the infection cannot establish itself anyway, and so the available antibodies will be sufficient to control the infection, independently of their decay rate. In particular, we see that the threshold $\mathcal{R}_0 =1$ is determined mainly by the value of the viral growth rate $a_1$ in this situation.

\subsubsection*{Influence of $a_2,a_4$}

\begin{figure}
\centering
$$\includegraphics[width=0.5\linewidth]{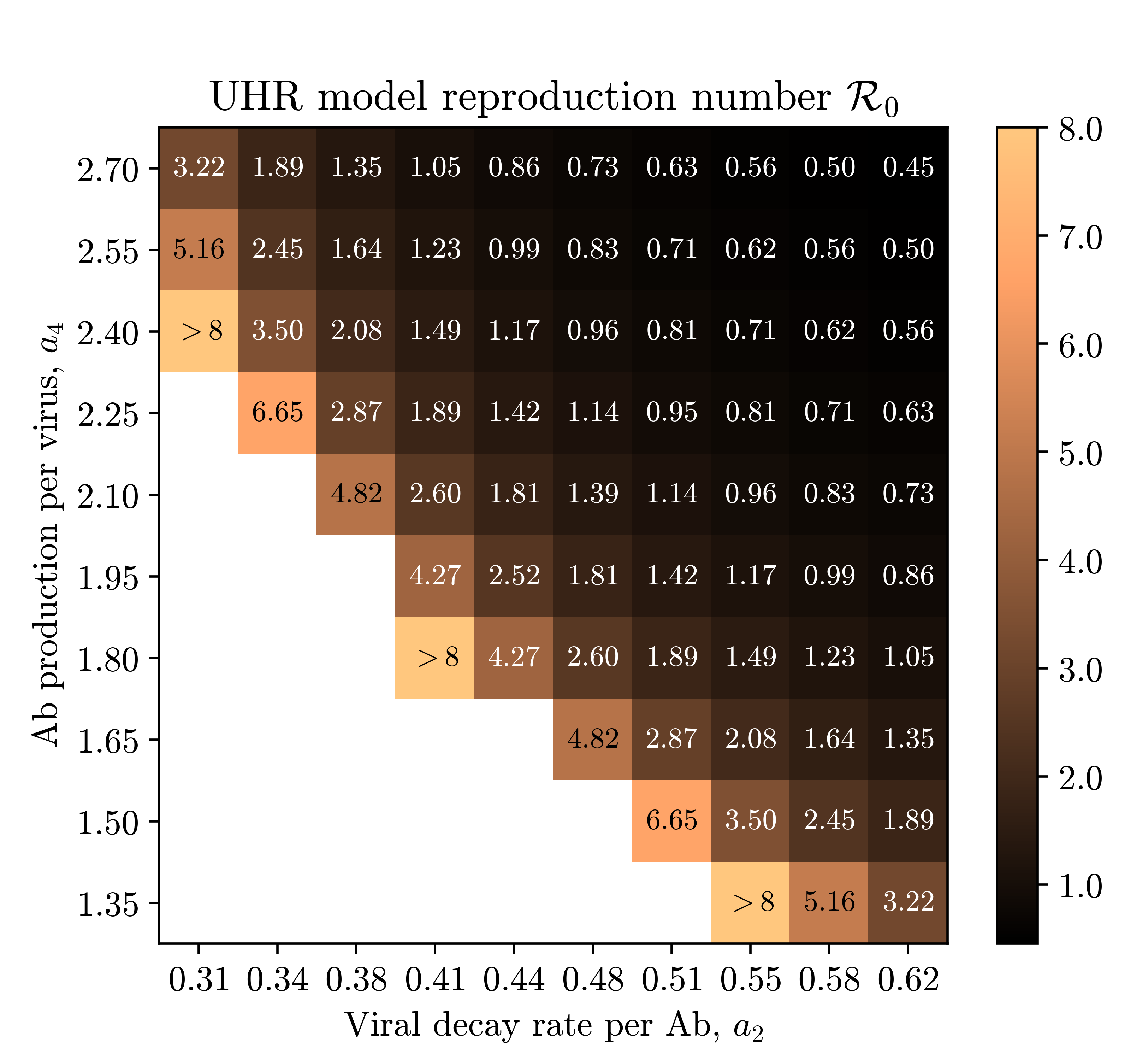}
\includegraphics[width=0.5\linewidth]{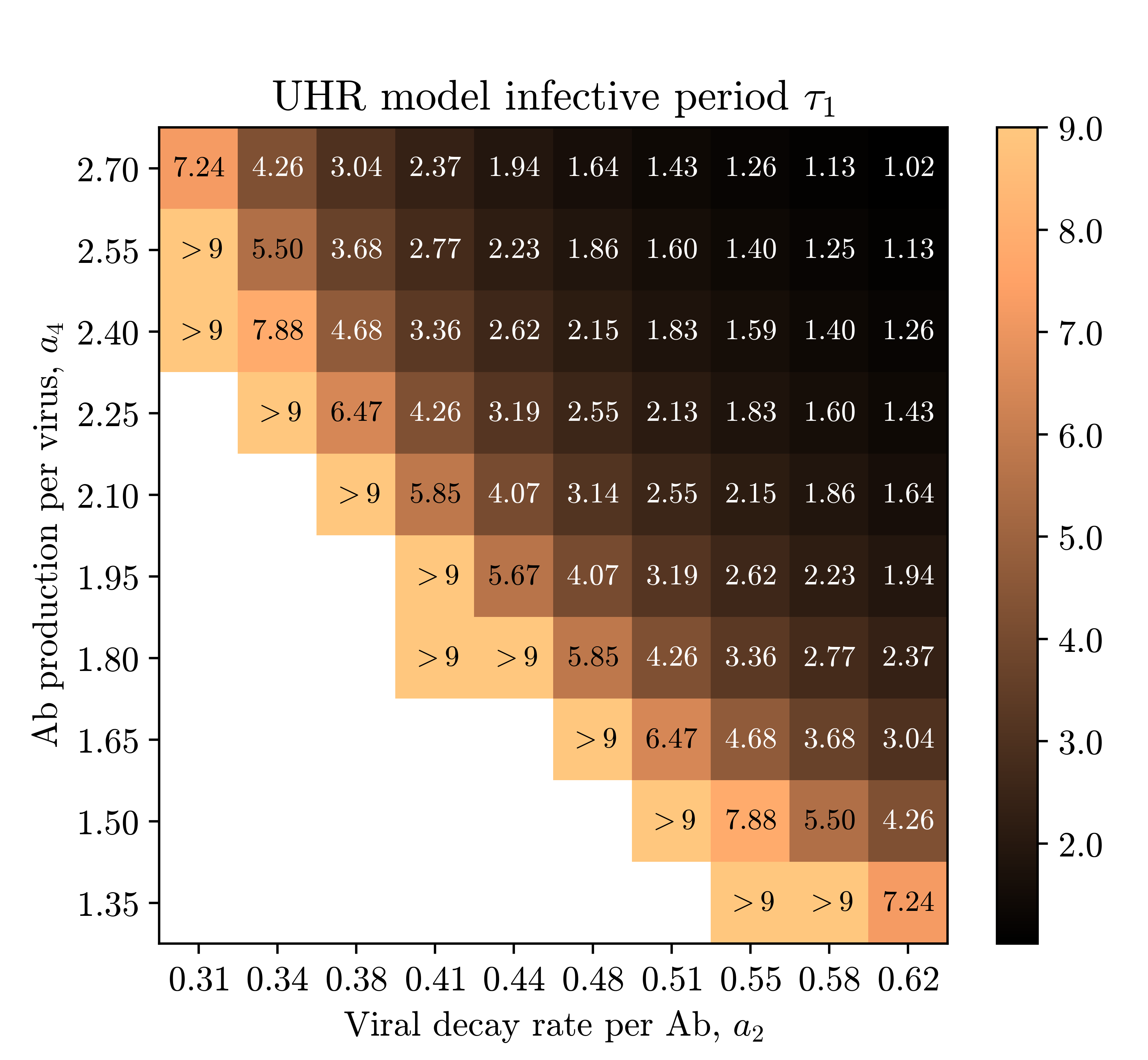}$$
\caption{$\mathcal{R}_0$ and $\tau_1$ for the UHR model as a function of $a_2$ and $a_4$, perturbed around the reference solution. The blank region corresponds to non-admissible $a_2,a_4$ values such that $\beta \not\in \RR$ (cf.~\eqref{eqns:beta}).}
\label{fig:R0a2a4}
\end{figure}

In Fig.~\ref{fig:R0a2a4}, we plot the values of $\mathcal{R}_0$ and $\tau_1$ as a function of the virus decay rate $a_2$ and the Ab production rate $a_4$. As expected, the value of $\mathcal{R}_0$ now decreases with both $a_2$ and $a_4$. 
The remark~\ref{rmk:tau} shows that the values in Fig.~\ref{fig:R0a2a4} are symmetric in $a_2,a_4$ and constant along the lines $a_2a_4 = C$. But this fact is a property of the particular within-host dynamics chosen. Other forms of dynamics with the same qualitative behavior might not share this precise symmetry property. Still, the experiment suggests that in similar models, the viral decay rate and the antibody production rate will have a similar influence on $\mathcal{R}_0$.

In both Figs.~\ref{fig:R0a1a3} and \ref{fig:R0a2a4}, we can confirm that the range of parameters $a_i$ which give $\mathcal{R}_0$ and $\tau_1$ in some desired range (that is, to fit with observations of a particular disease) appears to be narrow. This shows that the macroscopic dynamics is quite sensitive to the details of the within-host dynamics, at least in the context of the present class of models. In turn, this reinforces the idea that within-host dynamics can play an important role in the study and control of infectious diseases such as Dengue.

\subsubsection*{Influence of the antibody level of newly infected hosts, $y^\star$}

In both the full model and the uniform host response model, the population in the exposed compartment $E$ transitions to the infective compartment according to an exponential law, through the term $-\frac{1}{\tau_h}E(t)$.
The parameter $y^\star\in [0,y_0)$ represents the mean antibody level of individuals who reach the viral load $z_0$ at the onset of infection, after which they can transmit the infection. Recall that the UHR model was deduced from the full model \eqref{eqns:FullModel} precisely under the assumption that all newly infected have the same $y^\star$. Supposing that the within-host dynamics are described by our model, it should also be possible to find the value of $y^\star$ experimentally, in each application to a concrete disease.

It turns out that the value of $y^\star$ (or, more conveniently, the proportion $\alpha = y^\star / y_0 \in [0,1)$, which we use as a parameter) has a strong impact on the reproduction number $\mathcal{R}_0$. This is not unexpected, since we can see from the typical characteristic trajectory in Fig.~\ref{fig:Icharacterists} that if $\alpha$ is close to $1$, then the characteristic starting at $(z_0,y^\star)$ very quickly reaches the exit point $(z_0,y^+)$; in contrast, the characteristics originating from points with $y^\star \simeq 0$ stay in the infective region for much longer and therefore can contribute much more to new infections, increasing $\mathcal{R}_0$.

In Figs.~\ref{fig:ys1}-\ref{fig:ys4} we confirm this expectation. We consider values of $\alpha = y^\star/y_0$ ranging from $0.05$ to $0.95$ and simulate the UHR system using the reference parameters defined earlier, changing only $y^\star.$ Apart from the clear influence in the asymptotic values obtained (Figs.~\ref{fig:ys1},\ref{fig:ys2}), also the oscillating character of the solutions varies with $\alpha$, becoming more pronounced for smaller values of $\alpha$ (corresponding to larger $\mathcal{R}_0$).
\begin{figure}
    \centering
    \includegraphics[width=1\linewidth]{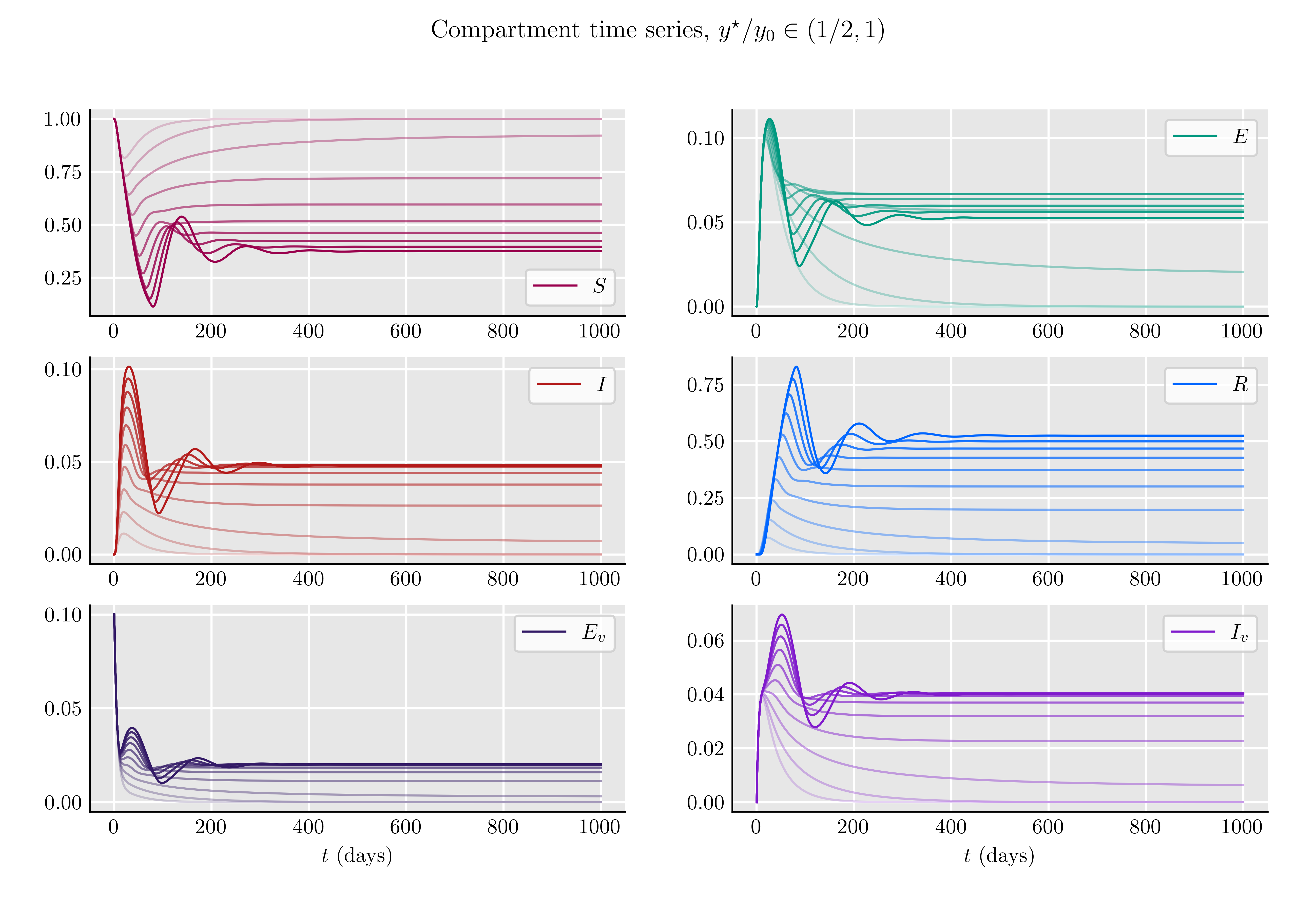}
    \caption{Compartment time series for varying $y^\star/y_0\in (1/2,1).$ Lighter curves correspond to higher values of $y^\star/y_0.$}   
    \label{fig:ys1}
\end{figure}
\begin{figure}
    \centering
    $$\includegraphics[width=0.7\linewidth]{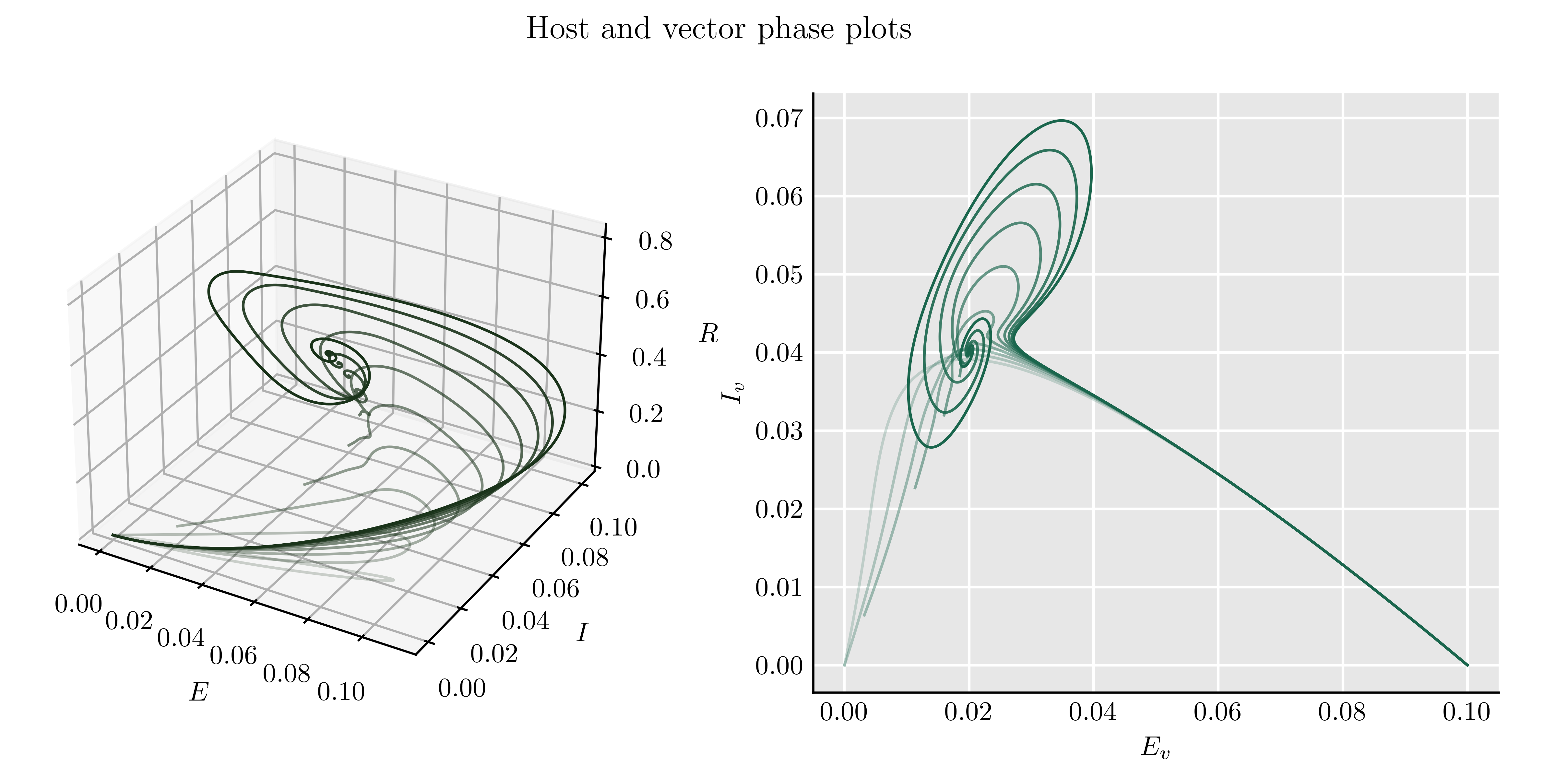}\includegraphics[width=0.3\linewidth]{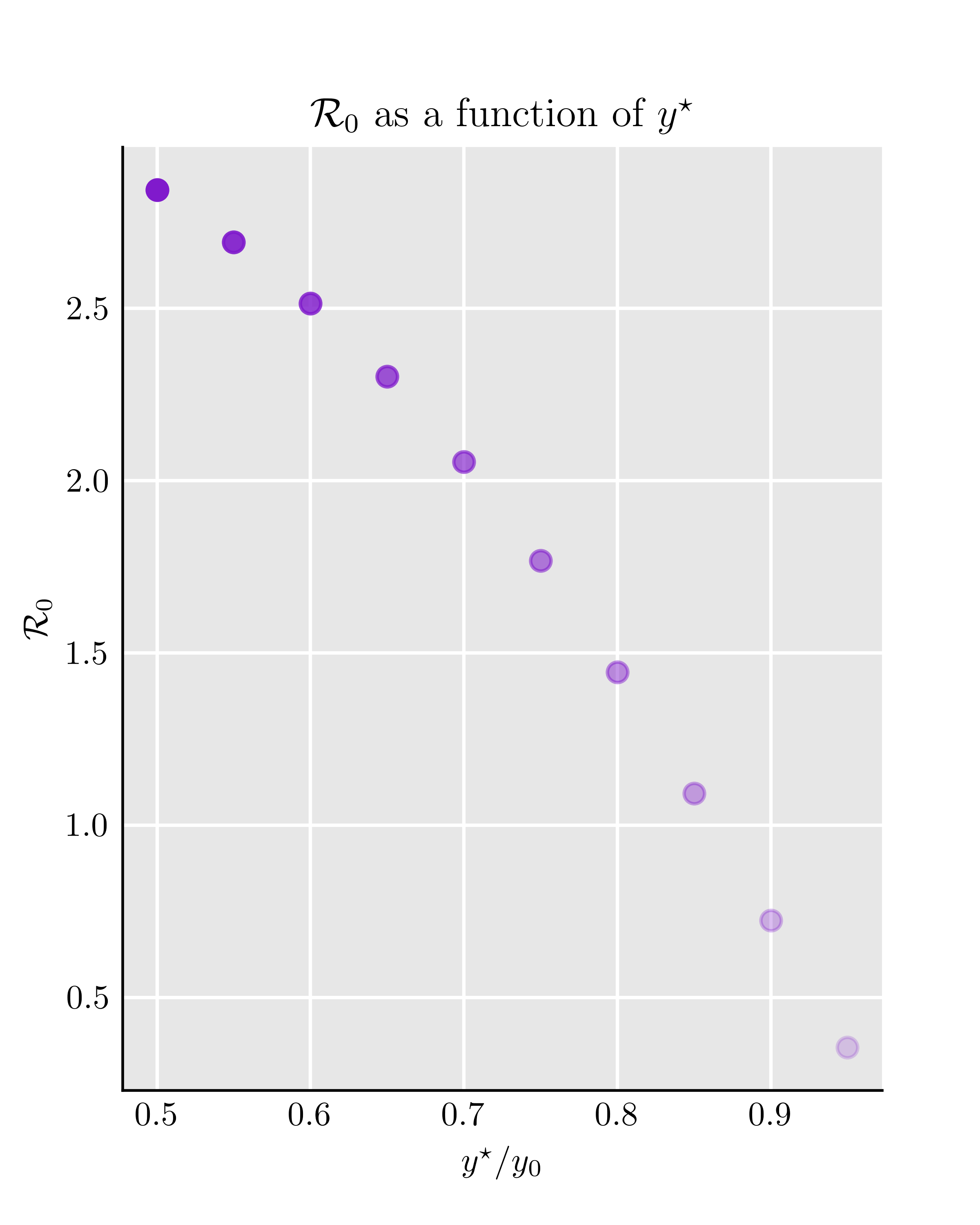}$$
    \caption{Phase trajectories of the solutions with $y^\star/y_0\in (1/2,1)$ (left, center). $\mathcal{R}_0$ as a function of $y^\star/y_0 \in (1/2,1)$ (right).}   
    \label{fig:ys2}
\end{figure}
\begin{figure}
    \centering
    \includegraphics[width=1\linewidth]{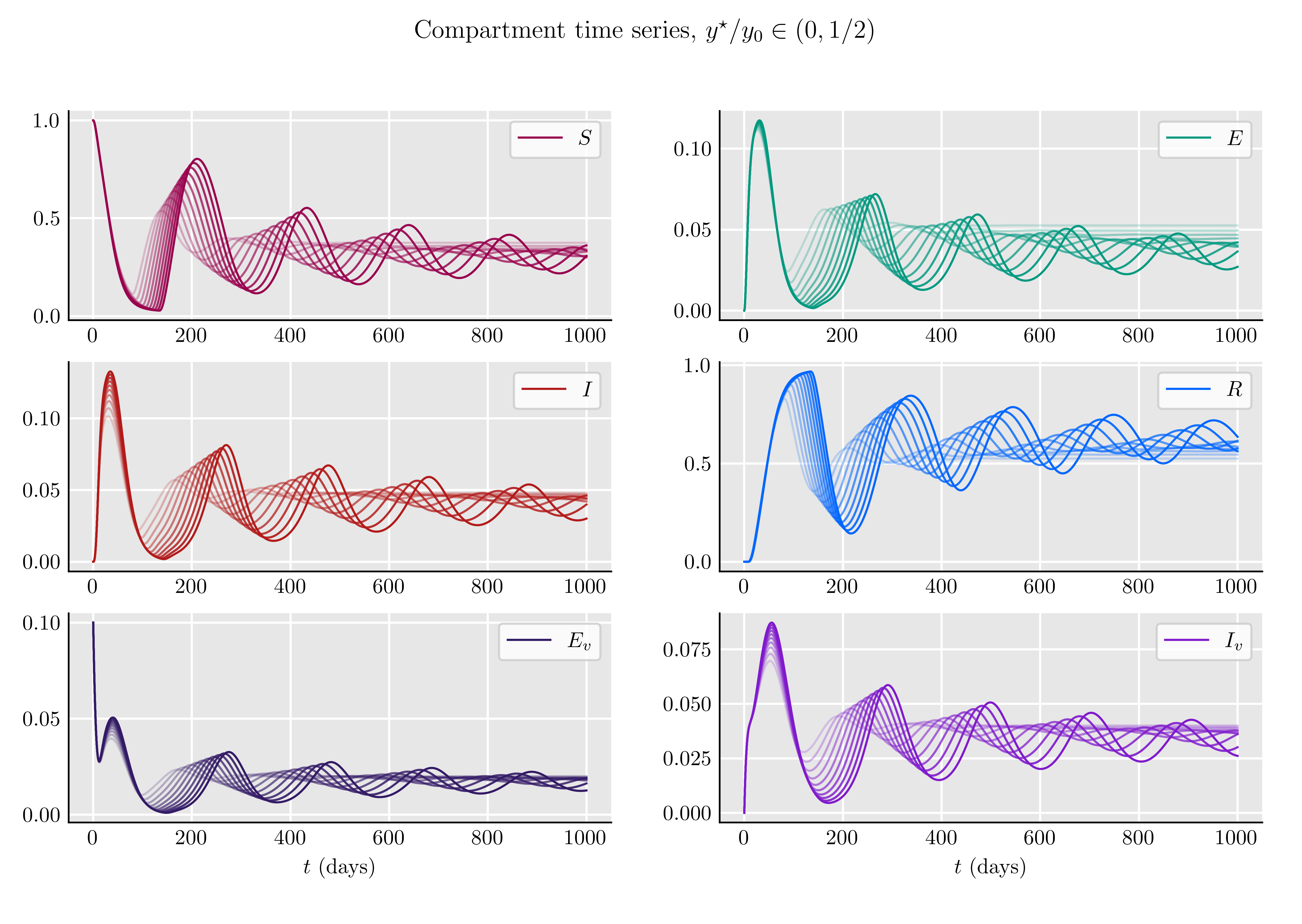}
    \caption{Compartment time series for varying $y^\star/y_0\in (0,1/2).$ Lighter curves correspond to higher values of $y^\star/y_0.$}   
    \label{fig:ys3}
\end{figure}
\begin{figure}
    \centering
    $$\includegraphics[width=0.7\linewidth]{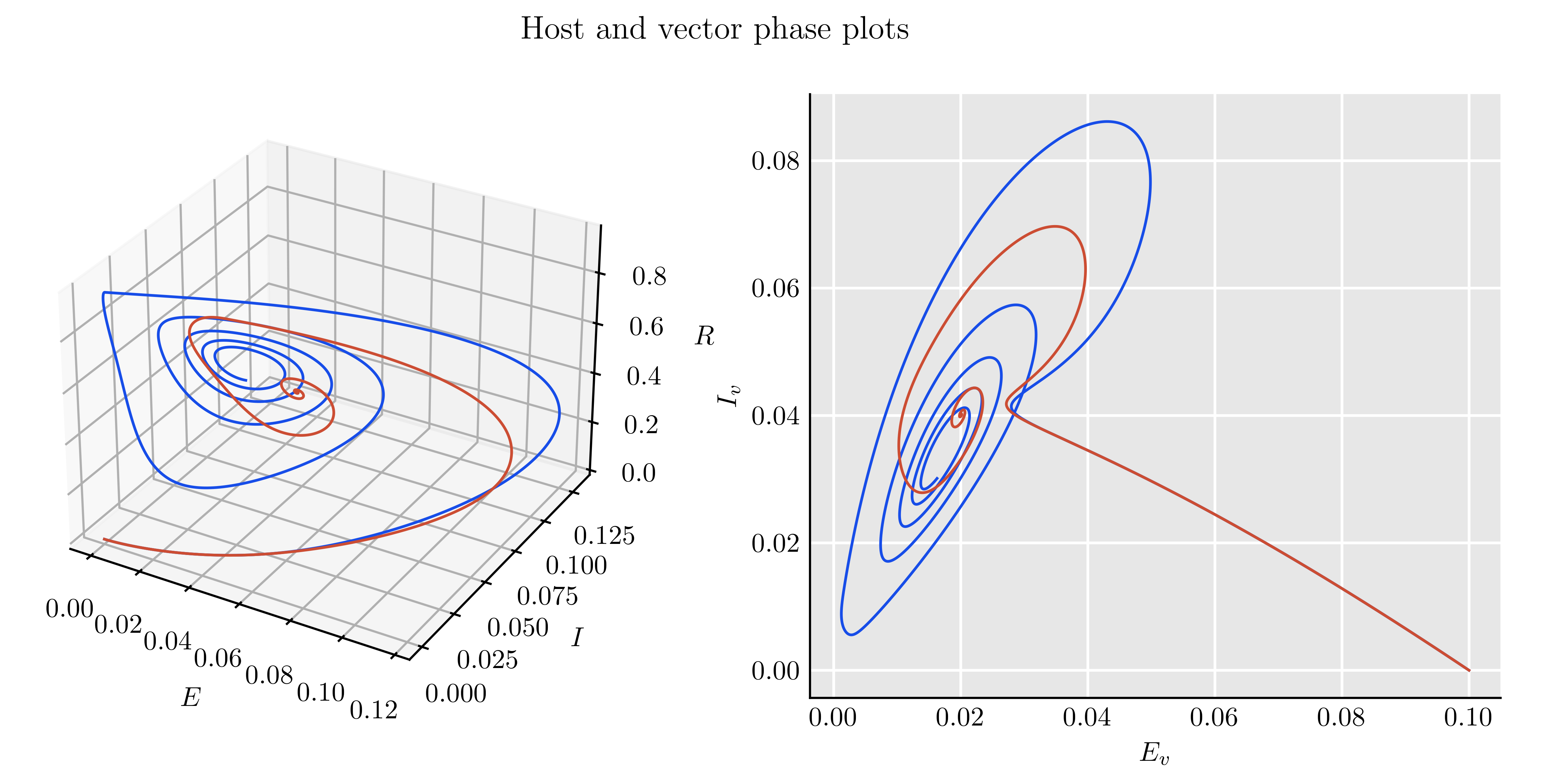}\includegraphics[width=0.3\linewidth]{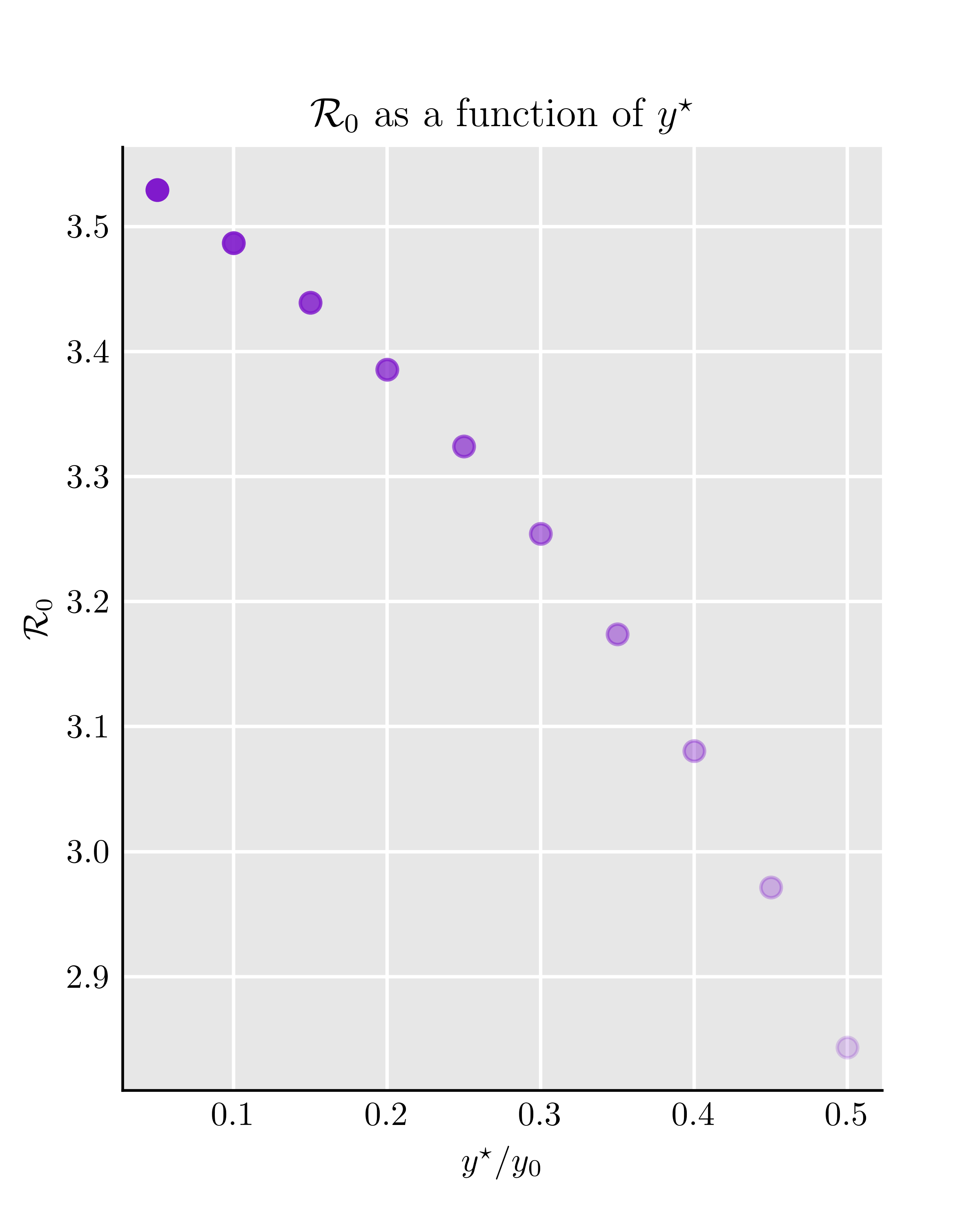}$$
    \caption{Phase trajectories of the solutions with $\alpha =y^\star/y_0\in (0,1/2)$ (left, center). For clarity, only $\alpha = 0.05$ and $1/2$ are shown.  $\mathcal{R}_0$ as a function of $y^\star/y_0 \in (0,1/2)$ (right).}   
    \label{fig:ys4}
\end{figure}

\subsection{Numerical investigation of $\mathcal{R}_0$ and the stability of the disease-free equilibrium}

We saw in Section \ref{sec:DFEstab} how a rigorous result linking the stability of the DFE and the value of $\mathcal{R}_0$ is challenging to obtain, even for the UHR model. In the next experiment, we show that, numerically at least, $\mathcal{R}_0 <1$ is equivalent to asymptotic stability of the DFE. To do this, we perform 100 simulations of the UHR model \eqref{eqns:delay} until $t_f=10^4$, each with a different value of $\mathcal{R}_0$, and plot the value of $I(t_f)$. In Fig.~\ref{fig:R0} we can see that $\mathcal{R}_0=1$ acts as a threshold between asymptotic stability and instability of the DFE. Since $\mathcal{R}_0$ is a derived parameter of the model, we achieved its variation by changing the value of $m$. 

We expect the $\mathcal{R}_0$ of the full system, \eqref{R_0Full}, to also serve as a threshold value for the stability of the DFE. Our experiments in this direction (omitted here for the sake of brevity) highlighted the previously mentioned difficulty in the numerical simulation of the transport equations. We performed for the full system the same series of simulations as described for the UHR system, varying $\mathcal{R}_0$. We observed a clear threshold value, just as in the UHR system, however this value was slightly above $\mathcal{R}_0=1$. As the mesh was refined, the threshold gets closer and closer to $\mathcal{R}_0=1$ from above, but the computation became unrealistic due to the fine mesh and the correspondingly small time step. In any case, the experiment strongly suggests that the threshold property also holds, but better numerical methods for the transport equation are needed to properly investigate this.

    % \begin{figure}[ht]
    %     \centering
    % \includegraphics[width=0.75\linewidth]{Figs/\mathcal{R}_0xSandR_FinalTime1000_ModeloPaper.png}
    % \caption{$\mathcal{R}_0$ versus endemic equilibrium, for a final time $T_f = 10^3$.}
    % \label{fig:R0}
    % \end{figure}

    \begin{figure}[ht]
        \centering
    \includegraphics[width=0.75\linewidth]{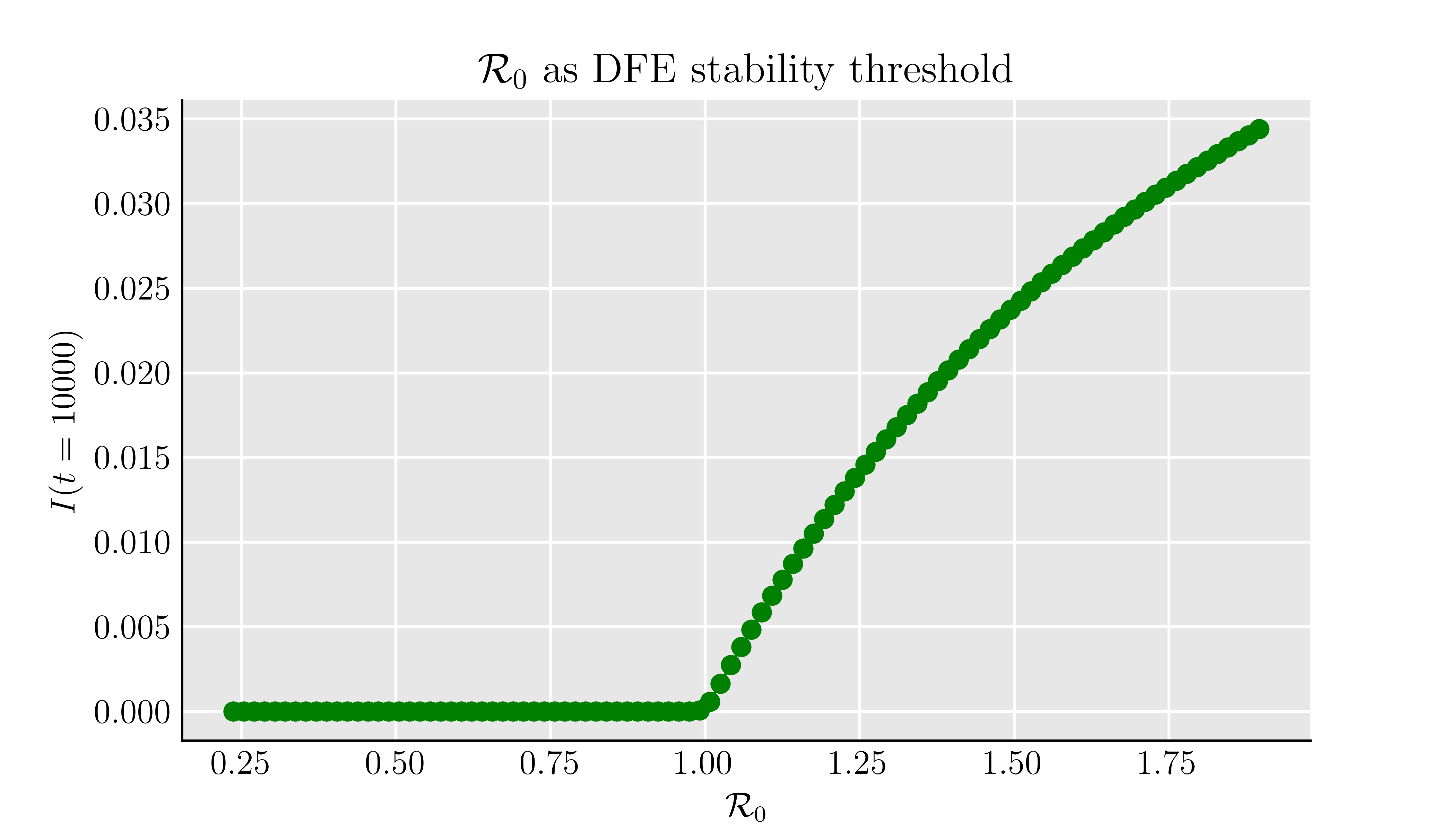}
    \caption{$\mathcal{R}_0$ versus asymptotic value of $I(t_f)$ at $t_f=10^4$.}
    \label{fig:R0}
    \end{figure}

%
%
%
%%%%%%%%%%%%%%%%%%%%%%%%%%%%%%%%%%%%%%%%%%%%%%%%%%%%%%%%%%%%%%%%%%
\section{Discussion}
\label{sec:Discussion}
%%%%%%%%%%%%%%%%%%%%%%%%%%%%%%%%%%%%%%%%%%%%%%%%%%%%%%%%%%%%%%%%%%

In this work, we introduced a host–vector model for vector-borne disease transmission that integrates within-host viral and immune dynamics into a population-level framework. By structuring the infected host population according to viral load and antibody level, the model captures the interaction between pathogen replication and the host immune response. This approach allows the model to represent the heterogeneity among infected individuals arising from differences in infection stage and immune status, which may play an important role in determining disease spread. In addition, the recovered population is structured by antibody level, reflecting the persistence and gradual waning of immunity after viral clearance and allowing the model to represent temporary protection against reinfection.

From a mathematical perspective, we established several fundamental properties of the system, including well-posedness, the characterization of characteristic curves, and the derivation of the basic reproduction number 
$\mathcal{R}_0$ together with the corresponding disease-free equilibrium. {We prove, both for the full model and the UHR model, that the existence of an endemic equilibrium is equivalent to $\mathcal{R}_0>1$. The reproduction number for the full model, defined in Prop.~\ref{prop:R0Full}, has a nontrivial form involving nonlocal terms.} To further explore the qualitative behavior of the system, we also considered a simplified formulation associated with a uniform host immune , which we called the Uniform Host Response system (or UHR system for short), which is mathematically simpler and allows to overcome challenges in the numerical simulation, thereby providing additional  insight into the dynamics of the model.

The numerical simulations illustrate how the interaction between viral dynamics, immune response, and vector transmission influences the course of the epidemic. In particular, the results highlight the sensitivity of the system to key epidemiological and immunological parameters. Overall, the proposed framework provides a new setting for investigating how individual-level infection dynamics may shape transmission patterns at the population level.

{The model has, of course, some limitations. For instance, numerical simulation of the full system (which accounts for variability in the host immune response) is challenging, as hyperbolic equations require high-resolution methods to yield useful numerical solutions. This limits the dynamical properties we could showcase for the full model, but can be overcome in future works by employing more sophisticated numerical methods. Also, the within-host modeling is phenomenological, in the sense that it is not derived from first principles. However, the model is built to allow other kinds of deterministic interactions between the antibody and virus, which can be encoded in the velocity terms in equation~\eqref{eqns:FullModel2}. Additionally, the numerical results we presented have shown how it is possible to narrow down the appropriate ranges for the parameters related to within-host dynamics, and to estimate how impactful each of them is on disease outcomes.}

A natural extension of the model would be to incorporate multiple dengue serotypes and account for the possibility of an exacerbated immune response during secondary infections with a different serotype, a phenomenon associated with antibody-dependent enhancement.

\section{Technical proofs}
\label{sec:proofs}
\subsection{Proof of Proposition \ref{prop:cons}}

% \begin{proof}
  %\todo[inline]{Prova}
  
  The system \eqref{eqns:FullModel} gives immediately
  \begin{equation*}
    \begin{aligned}
        \frac{d}{dt}& \Big(  S(t) + E(t) + \int_0^{\infty}\!\!\int_{z_0}^{\infty} {I(t,z,y)}  \,dz dy + \int_{y_0}^{+\infty} R(t,y) \,dy  \Big) 
        = a_5y_0 R(t,y_0) - \frac{1}{\tau_h}E(t) 
        \\
        &\quad - \int_0^{\infty}\!\!\int_{z_0}^{\infty}\frac{\partial }{\partial z} \big(( a_1 z - a_2 y) I(t,z,y)  \big) + \frac{\partial }{\partial y} \big( (-a_3 y  + a_4 z ) I(t,z,y)  \big)  \,dz dy 
        \\
        & \quad + \int_{y_0}^{+\infty} \frac{\partial }{\partial y} \left( a_5 y R(t,y) \right) \,dy - \int_{y_0}^{+\infty} I(t,z_0, y) \bigl( a_1 z_0 - a_2 y \bigr) \,dy. 
    \end{aligned}
  \end{equation*}
   Assuming classical solutions, the boundary conditions \eqref{eqns:BCI} give 
  \begin{equation*}
    \begin{aligned}
        & - \int_0^{\infty}\!\!\int_{z_0}^{\infty}\frac{\partial }{\partial z} \big(( a_1 z - a_2 y) I(t,z,y)  \big) + \frac{\partial }{\partial y} \big( (-a_3 y  + a_4 z ) I(t,z,y)  \big)  \,dz dy
        \\
        & = \Big(\int_0^{y_0} + \int_{y_0}^{\infty}\Big) (a_1 z_0 - a_2 y) I(t,z_0,y) \,dy.
    \end{aligned}
  \end{equation*}
  The second term cancels, while the first term is equal to
  \begin{equation}
      \begin{aligned}
          \int_0^{y_0} (a_1 z_0 -a_2y) \frac{E(t) g(y)}{(a_1 z_0 - a_2 y^\star) \tau_h}\,dy 
      \end{aligned}
  \end{equation}
  which is simply
  \begin{equation}
      \frac{E(t)}{\tau_h},
  \end{equation}
  from the definition of $y^\star$ and $\int_0^{y_0} g(y) \,dy = 1.$ Similarly, the terms involving $R$ easily cancel,
  which concludes the proof.
% \end{proof}

\subsection{Proof of Proposition \ref{prop:R0Full}}

%\begin{proof}[Proof of Proposition \ref{prop:R0Full}]
    From the expressions \eqref{equi-delay}, it is clear that the endemic equilibrium exists and has components in $(0,1)$ if, and only if, $\mathcal{R}_0 >1.$ Therefore it remains to show that setting $E^*, E_v^*, I_v^*$ according to \eqref{equi-delay}, then defining $I^*(z,y)$ by \eqref{Istar} and $R^*(y)$ by \eqref{Rstar}, we get a solution to \eqref{eqns:EquiFull}.

    Using the last three equations of \eqref{eqns:EquiFull}, we find after some straightforward computations,
    % \begin{equation*}
    %     \begin{aligned}
    %         E^* \Big( 1+ \frac{b \beta_{vh}m\tau_h}{1+\tau_v\mu_v}\Big(1+ \frac{1}{E^*}\iint I^*+\frac{1}{E^*}\int R^*\Big)\Big) & =  \frac{b \beta_{vh}m\tau_h}{1+\tau_v\mu_v} - \frac{\mu_v}{b\frac{1}{E^*}\iint I^*\beta_{hv}}
    %     \end{aligned}
    % \end{equation*}
    % and then
    \begin{equation}\label{Estar1}
        \begin{aligned}
            E^* = \frac{\mu_v \tau_h\Big( \frac{b^2\beta_{vh}m\frac{\tau_h}{E^*}\iint I^*\beta_{hv}}{\mu_v(1+\tau_v\mu_v)} -1\Big)}{\mu_v \frac{b^2\beta_{vh}m\frac{\tau_h}{E^*}\iint I^*\beta_{hv}}{\mu_v(1+\tau_v\mu_v)}\Big(\tau_h + \frac{\tau_h}{E^*}\iint I^* + \frac{\tau_h}{E^*}\int R^* \Big)}.
        \end{aligned}
    \end{equation}
    Comparing \eqref{Estar1} with \eqref{equi-delay},\eqref{R_0Full}, and \eqref{defs1}, we see that the result will be proved if we show that
    \begin{equation}
        \begin{aligned}
            &\frac{\tau_h}{E^*}\iint I^* =\Tcal_1 ,
            \\
            &\frac{\tau_h}{E^*}\int R^* = \Tcal_2,
            \\
            &\frac{\tau_h}{E^*}\iint I^*\beta_{hv} = \Tcal_1[\beta_{hv}],
        \end{aligned}
    \end{equation}
    where $\Tcal_1,\Tcal_2,\Tcal_1[\beta_{hv}]$ are defined in \eqref{defs1}.
    Consider $\Tcal_1$ first. We have
    \begin{equation}\label{I01a}
        \begin{aligned}
            \frac{\tau_h}{E^*}\int_0^\infty \!\!\int_{z_0}^\infty I^*(z,y) \,dzdy &= \frac{\tau_h}{E^*} \int_0^{y_0} \!\! \int_0^{\tau_1(y')} I^*(z(\tau,y'),y(\tau,y')) |J\Phi| \,d\tau dy',
        \end{aligned}
    \end{equation}
    where $\Phi(\tau,y') = (z(\tau,y'),y(\tau,y'))$ is the characteristic flow and $\tau_1$ is defined in Prop.~\ref{prop:y}. One computes 
    \begin{equation}\label{I02}
        |J\Phi(\tau,y')| = e^{(a_1-a_3) \tau} V_z(z_0,y').
    \end{equation}On the other hand, integrating along the characteristic which reaches $(z,y)$ (see \eqref{Istar}), and using the boundary condition \eqref{eqns:BCI}, we have that
    \begin{equation}\label{I03}
    \begin{aligned} 
        I^*(z(\tau,y'),y(\tau,y')) &= I^*(z(0,y'),y(0,y')) e^{(a_3-a_1) \tau}
        \\
        & =I^*(z_0,y') e^{(a_3-a_1) \tau}
        \\
        & = \frac{E^*}{\tau_h}\frac{g(y')}{ V_z(z_0,y^\star)} e^{(a_3-a_1) \tau}.
    \end{aligned}
    \end{equation}
    Using \eqref{I02},\eqref{I03} in \eqref{I01a}, we get
    \begin{equation}\label{I04}
        \begin{aligned}
            \frac{\tau_h}{E^*} \int_0^{y_0} \!\! \int_0^{\tau_1(y')} I^*(z(\tau,y'),y(\tau,y')) |J\Phi| \,d\tau dy'& = \int_0^{y_0}  \!\! \int_0^{\tau_1(y')}\frac{g(y')V_z(z_0,y')}{V_z(z_0,y^\star)}\,d\tau dy'
            \\
            & =\int_0^{y_0} \frac{g(y')V_z(z_0,y')}{V_z(z_0,y^\star)}\tau_1(y') \,dy'
            \\
            & =  \frac{\int_{0}^{y_0}g(y)(a_1z_0 -a_2y) \tau_1(y) \,dy}{\int_{0}^{y_0}g(y)(a_1z_0 -a_2y) \,dy} = \Tcal_1.
        \end{aligned}
    \end{equation}
    For $\Tcal_1[\beta_{hv}]$ we proceed in the same way, but having in mind the presence of $\beta_{hv}(z)$. We get
    \begin{equation}\label{T01}
        \begin{aligned}
            \frac{\tau_h}{E^*}\int_0^\infty \!\!\int_{z_0}^\infty I^*(z,y) \beta_{hv}(z) \,dzdy &= \frac{\tau_h}{E^*} \int_0^{y_0} \!\! \int_0^{\tau_1(y')} I^*(z(\tau,y'),y(\tau,y')) \beta_{hv}(z(\tau,y')) |J\Phi| \,d\tau dy',
            \\
            &= \int_0^{y_0}  \!\! \int_0^{\tau_1(y')}\frac{g(y')V_z(z_0,y')}{V_z(z_0,y^\star)}\beta_{hv}(z(\tau,y'))\,d\tau dy'
            \\
            & = \frac{\int_{0}^{y_0}g(y)(a_1z_0 -a_2y) \Big[\int_{0}^{\tau_1(y)} \beta_{hv}(z(\tau,y)) d\tau\Big] \,dy}{\int_{0}^{y_0}g(y)(a_1z_0 -a_2y) \,dy} = \Tcal_1[\beta_{hv}].
        \end{aligned}
    \end{equation}
    For the $\Tcal_2$ term, we find by integration of the $R^*$ equation
    \begin{equation*}
        \begin{aligned}
            a_5 y R^*(y) = - \int_y^\infty I^*(z_0,y') V_z(z_0,y') \,dy',
        \end{aligned}
    \end{equation*}
    whence
    \begin{equation}\label{Y01}
        \begin{aligned}
            a_5 \int_{y_0}^\infty R^*(y) \,dy &= -\int_{y_0}^\infty\!\! \int_y^\infty \frac{1}{y}I^*(z_0,y') V_z(z_0,y') \,dy'dy
            \\
            &= -\int_{y_0}^\infty \!\! \int_{y_0}^{y'} \frac{1}{y} I^*(z_0,y') V_z(z_0,y') \,dydy'
            \\
            &= -\int_{y_0}^\infty I^*(z_0,y')V_z(z_0,y') \ln(y'/y_0) \,dy'.
        \end{aligned}
    \end{equation}
    To obtain the expression of $\Tcal_2$ from \eqref{Y01}, we must write the integral in $y'\in (0,y_0)$. We proceed as follows. Given a function $\overline a(y)$ for $y\in (y_0,y^+(0))$, we can consider the function $a(z,y)$ defined on $A$ which is constant along the (backward) characteristic curves from $(z_0,y^+)$ to the corresponding $(z_0,y') \in (0,y_0),$ and such that $a(z_0,y) = \overline{a}(z_0)$ for $y> y_0$. By definition of the characteristics, we have $(V_z,V_y)\cdot \nabla a =0$. Then we have $\mathrm{div}((V_z,V_y)I^*a) = \mathrm{div}((V_z,V_y)I^*)a + I^* (V_z,V_y)\cdot \nabla a = 0.$
    By the Divergence Theorem, we find
    \begin{equation}
        \begin{aligned}
            0 & = \int_0^\infty \!\!\int_{z_0}^\infty \mathrm{div}((V_z,V_y)I^*a) \,dzdy
            \\
            &= \int_0^{y_0} V_z(z_0,y') I^*(z_0,y') a(z_0,y') \,dy' - \int_{y_0}^\infty V_z(z_0,y^+) I^*(z_0,y^+) a(z_0,y^+) \,dy^+.
        \end{aligned}
    \end{equation}
    Since $a$ is constant along each characteristic, we have $a(z_0,y') =a(z_0,y^+(y')) = \overline{a}(y^+(y')).$ Thus,
    \begin{equation}
        \begin{aligned}
            \int_{y_0}^\infty V_z(z_0,y^+) I^*(z_0,y^+) \overline{a}(y^+) \,dy^+ &=\int_0^{y_0} V_z(z_0,y') I^*(z_0,y') \overline{a}(y^+(y')) \,dy'.
        \end{aligned}
    \end{equation}
    With $\overline{a}(y)= \ln(y/y_0)$, we get from \eqref{Y01} that 
    \begin{equation*}
        \begin{aligned}
            \int_{y_0}^\infty R^*(y) \,dy &= \frac{1}{a_5}\int_0^{y_0} V_z(z_0,y') I^*(z_0,y') \ln(y^+(y')/y_0) \,dy'
            = \Tcal_2,
        \end{aligned}
    \end{equation*}
    after using the boundary condition \eqref{eqns:BCI} again. This completes the proof of Proposition~\ref{prop:R0Full}.
% \end{proof}

%
%
%
%%%%%%%%%%%%%%%%%%%%%%%%%%%%%%%%%%%%%%%%%%%%%%%%%%%%%%%%%%%%%%%%%%
\section*{Acknowledgements}
The authors were partially supported by the CAPES/MATH-Amsud grant ``StrucMetaMoC" nr. 88881.179477/2025-01, and by the CNPq Universal grant nr. 401905/2025-0. SA was partially supported by FAPERJ grants nr.~E-26/210.037/2024 and 260003/020457/2025, ad CNPq grant 312407/2023-8.
PA thanks Prof. Bernardo F.P. Costa for helpful discussions.

\bibliographystyle{plain}       
\bibliography{viralload}
%%%%%%%%%%%%%%%%%%%%%%%%%%%%%%%%%%%%%%%%%%%%%%%%%%%%%%%%%%%%%%%%%%

\end{document}